\documentclass{amsart}

\usepackage[foot]{amsaddr}
\usepackage{etex}
\usepackage[usenames,dvipsnames]{pstricks} 
\usepackage{epsfig}
\usepackage{graphicx,color}
\usepackage{geometry}
\usepackage[all]{xy}
\usepackage{amssymb}
\usepackage{cite}
\usepackage{fullpage}
\xyoption{poly}
\usepackage{url}
\usepackage{hyperref}
\numberwithin{equation}{subsection}
\numberwithin{figure}{subsection}
\usepackage{longtable}

\usepackage{tikz}
\usepackage{tikz-cd}
\usetikzlibrary{decorations.pathmorphing}
\newtheorem{introtheorem}{Theorem}

\newtheorem{theorem}{Theorem}[subsection]
\newtheorem{lemma}[theorem]{Lemma}
\newtheorem{proposition}[theorem]{Proposition}
\newtheorem{corollary}[theorem]{Corollary}
\theoremstyle{definition}
\newtheorem{definition}[theorem]{Definition}

\newtheorem{remark}[theorem]{Remark}

\newtheorem*{question*}{Question}

\newtheorem*{steps*}{Answer/steps}

\newtheorem*{progress*}{Progress}
\newtheorem*{example*}{Example}

\newtheorem{conditions}[theorem]{Conditions}

\newtheorem*{remark*}{Remark}
\newtheorem*{remarks*}{Remarks}
\newtheorem*{definition*}{Definition}

\usepackage{calrsfs}

\usepackage[T1]{fontenc}
\usepackage{textcomp}
\usepackage{times}
\usepackage[scaled=0.92]{helvet}

\newcommand{\QQ}{\mathbb{Q}}
\newcommand{\QQb}{{\overline{\QQ}}}
\newcommand{\Q}{\mathbb{Q}}

\newcommand{\ZZ}{\mathbb{Z}}

\newcommand{\FF}{\mathbb{F}}
\newcommand{\PP}{\mathbb{P}}

\DeclareMathOperator{\dlog}{{dlog}}

\DeclareMathOperator{\Aut}{Aut}

\DeclareMathOperator{\Gal}{Gal}
\DeclareMathOperator{\sgn}{sgn}
\DeclareMathOperator{\supp}{supp}
\DeclareMathOperator{\Spec}{Spec}

\begin{document}

\title{Dynamical Belyi maps and arboreal Galois groups}
\author{Irene I.~Bouw}
\address{Ulm University, Ulm, Germany}
\email{irene.bouw@uni-ulm.de}
\author{\"{O}zlem Ejder}
\address{Colorado State  University, Fort Collins, CO, USA}
\email{ejder@math.colostate.edu}
\author{Valentijn Karemaker}
\address{Utrecht University, Utrecht, The Netherlands and Stockholm University, Stockholm, Sweden}
\email{V.Z.Karemaker@uu.nl}

\begin{abstract}
We consider a large class of so-called dynamical Belyi maps and study
the Galois groups of iterates of such maps. From the combinatorial
invariants of the maps, we construct a useful presentation of the geometric
Galois groups as subgroups of automorphism groups of regular trees, in
terms of iterated wreath products. Using results on the reduction of dynamical Belyi maps modulo certain primes, we obtain results on the corresponding arithmetic Galois groups of iterates. These lead to results on the behavior of the arithmetic Galois groups under specialization, with applications to dynamical sequences.
\end{abstract}

\date{}
\maketitle

\section*{Introduction}
 
 Let $f:\PP^1_K\to \PP^1_K$ be a rational map defined over a number 
 field $K$. The Galois theory of the iterates $f^n=f\circ \cdots \circ
 f: \PP^1_K\to \PP^1_K$ was first studied in the work of Odoni
 (\cite{Odoni-iterates} ), and has applications both in number theory
 and in arithmetic dynamics.  
  Specializing the iterates $f^n$ at a $K$-rational place $a\in \PP^1_K$, we obtain a  tower of number
 fields $(K_{n,a})_{n \geq 1}$. Many recent papers  study the question of the
 distribution of the  primes of $K$ that ramify in this tower.

We denote the Galois group of $f^n$ by $G_{n, K}= G_{n,K}(f)$ and the
Galois group of its specialization at $a$ by
$G_{n,a}=G_{n,a,K}(f)$. For places $a$ avoiding a finite subset of
$\PP^1_K$, the limit $G_{\infty, a}=\varprojlim_n G_{n, a}$ acts on
the infinite $d$-ary regular tree $T_\infty$, where $d=\deg(f)$. We
obtain an \emph{arboreal representation}, and hence a map of
$G_{\infty, a}$ to the automorphism group $\Aut({T}_\infty)$ of the
tree; this is why we call the group $G_{\infty,a}$ an \emph{arboreal
  Galois group}.  A paper by Jones (\cite{JonesSurvey} ) provides a
survey of the theory of arboreal representations.  A central question
is to characterize when $G_{\infty,a}$ embeds in $\Aut({T}_\infty)$ as
a subgroup of finite index (\cite[Question 1.1]{JonesSurvey} ). As
Jones discusses, this question may be considered as an analog of
Serre's open image theorem.  We take a different perspective and study
a class of rational maps for which the image of $G_{\infty,a}$ has
infinite index in $\Aut({T}_\infty)$.

The maps we study in this paper are Belyi maps $f:\PP^1_K\to \PP^1_K$
with exactly $3$ ramification points, which we assume to be
$0,1,\infty$. We normalize the coordinate on the target projective
line such that the ramification points are fixed by $f$. A map $f$
satisfying these properties is called a \emph{normalized (dynamical)
  Belyi map} and is completely determined by its (combinatorial) type
(Definition \ref{def:type}), which tabulates the ramification
indices. In particular, all maps may be defined over $K=\QQ$
(Proposition \ref{prop:rigid}). All normalized Belyi maps are
so-called post-critically finite (PCF) maps, since the forward orbit
of each ramification point is preperiodic.  A result of Jones and Pink
(\cite[Theorem 3.1]{JonesSurvey} ) states that the index
$[\Aut({T}_\infty):G_{\infty,a}]$ is infinite for PCF maps, and hence
for ours, as well.

The arboreal Galois group $G_{\infty, a}$ of a specialization of a PCF
map is a mysterious group, which is hard to describe in general.  The
case that $f$ is a polynomial of degree $\deg(f)=2$ has been
extensively studied, see e.g.~\cite{census}, \cite{pink}.  In
\cite{BFHJY}, the authors consider the case of the cubic polynomial
$f(x)=-2x^3+3x^2$, providing the first completely explicit result on
its Galois theory. This cubic polynomial is the easiest example of a
normalized Belyi map.

A related problem, which is very interesting but difficult in general,
is to determine the primes that ramify in the tower of number fields
$(K_{n,a})_{n \geq 1}$.  For PCF maps $f$, it is known that only
finitely many primes ramify in the whole tower $(K_{n,a})_{n \geq
  1}$. This was proven in \cite[Theorem 1.1]{aitken} in the case that
$f$ is a polynomial and in \cite[Theorem 1]{CulFar} in the general
case; see also \cite[Theorem 3.2]{JonesManes}.

\bigskip
The goal of this paper is to study the Galois group $G_{\infty, a}$
and the primes ramifying in the corresponding tower of number fields
for normalized Belyi maps.  The class of normalized Belyi maps we
study in this paper contains maps of arbitrarily large degree, which
are not necessarily polynomial. Explicit expressions exist for these
maps; Proposition \ref{prop:Belyifam} gives a sample. However, we do
not rely on these to prove general results for this class of
maps. 

Instead, we systematically use the Galois-theoretical
characterization of these maps (Section \ref{sec:belyi}) and exploit the description of their Galois groups as an iterated wreath product. 
In particular, the combinatorial description of the monodromy of $f$ (in terms of its combinatorial type) yields a combinatorial description of the monodromy of its iterates, which enables us to explicitly determine the geometric Galois groups $G_{n,\QQb}$ for all $n \geq 1$ in Theorem~\ref{thm:G2alt}. 
Furthermore, the product discriminant that we introduce in Section~\ref{sec:descent} allows us to study the arithmetic Galois groups $G_{n,\QQ}$ and to make a comparison between the arithmetic and geometric Galois groups in Section~\ref{sec:descent}.
Finally, in analyzing the properties of the specialized Galois groups $G_{n,a}$ (Theorem~\ref{G2aG2}) our full understanding of the ramification structure of iterates of dynamical Belyi maps plays a key role. Another important ingredient is that the reduction behavior of a normalized Belyi map $f$, like its monodromy, can also be completely expressed in terms of its combinatorial type, often yielding explicit and easy to apply criteria for good and
bad reduction. Normalized Belyi maps were already studied in the context of arithmetic dynamics in our previous paper \cite{WINE2}, where we proved a first result on the reduction of normalized Belyi maps (Proposition \ref{prop:Belyiredcrit}). In the current paper, we
exploit this approach more fully to study the Galois theory of the towers $(K_{n,a})_{n \geq 1}$ for normalized Belyi maps.  

Normalized Belyi maps form a very rich source of PCF maps.  The
techniques for studying their iterates presented here illustrate that
the whole class of maps may be analyzed by general methods, which
makes the class accessible for future applications in arithmetic
dynamics.  Just to mention one possible application: Remark
\ref{rem:tower} sketches how our results may be used to explicitly
construct infinite towers of number fields branched over an explicit
finite set of primes. Our knowledge on the reduction of a normalized
Belyi map yields a much smaller set containing the branched primes
than could be expected from previous work.

\bigskip
We now describe our results more precisely.  Our first main result
completely describes the geometric groups $G_{n, \QQb}$ as subgroups of
$\mathrm{Aut}(T_n)$ for any $n \geq 1$, excluding two exceptional
cases in small degree. The geometric group $G_{n, \QQb}$ is the Galois group of
$f^n:\PP^1_\QQb\to \PP^1_\QQb$ considered as map defined over $\QQb$.
 
\begin{introtheorem}[Theorem \ref{thm:G2alt}]\label{intro1}
Let $f$ be a normalized Belyi map and let $E_n \subseteq
\mathrm{Aut}(T_n)$ be the subgroup defined in Definition
\ref{def:En}. With two exceptions  the group
$G_{n,\QQb} $ is either isomorphic to $E_n$ or to the $n$-fold
iterated wreath product of $A_d$ with itself. The case distinction
only depends on the type of $f$ and is independent of $n$.
\end{introtheorem}

The two possibilities for $G_{n, \QQb}$ can be described totally
explicitly. For example, there is an easy formula for its index in
$\Aut(T_n)$ (Lemma \ref{lem:En}(1)). 

In general, the
Galois group $G_{n, \QQ}$ of $f^n$, considered as map defined over
$\QQ$, is strictly larger than $G_{n, \QQb}$. In Corollaries~\ref{cor:condsgn} and \ref{cor:descent} and Remark~\ref{rem:descent} we describe exactly when $G_{n,\QQ}=G_{n, \QQb}$ for all $n\geq
1$.

Hilbert's Irreducibility Theorem implies that $G_{n, a}=G_{n, \QQ}$
for $a$ in a dense open subset of $\PP^1_\QQ$. We give an explicit
criterion on $a$ and the type of the normalized dynamical Belyi map
$f$ that guarantees that $G_{n, a}=G_{n, \QQ}=G_{n, \QQb}=E_n$ for all $n\geq 1$:

\begin{introtheorem}[Theorem \ref{G2aG2}]\label{intro2}
Let $a \in \mathbb{P}^1(\QQ)\setminus \{0,1,\infty\}$ be chosen such
that Conditions \ref{conditions} are satisfied for some choice of distinct
rational primes $p, q_1, q_2, q_3$. Assume that $G_{n, \QQ}=G_{n,
  \QQb}$ for all~$n\geq~1$.  Then $G_{n,a} \simeq G_{n,\QQ}$ for all $n \geq
1$.
\end{introtheorem}

While the conditions we impose are somewhat involved, it is
not hard to find instances where all conditions are satisfied; Remark
\ref{rem:tower} gives a sample.

Theorem \ref{intro2} is a generalization of the main result (Theorem
1.1) of \cite{BFHJY}, in which the authors consider the special case
that $f(x)=-2x^3+3x^2$ is a concrete polynomial of degree $3$. Our Conditions
\ref{conditions} generalize Condition $(\dagger)$ in that paper. The
general strategy for proving our results is similar to that of
\cite{BFHJY}. However, the details of the proofs are quite different.
The authors of \cite{BFHJY} can rely on ad hoc arguments in terms of
the explicit polynomial $f(x)$ that do not extend
directly. In our much more general setting, we exploit the
Galois-theoretic properties of normalized Belyi maps and the
group-theoretic properties of arboreal Galois groups as iterated
wreath products instead.

Our last main result is an application to arithmetic dynamics.  Let
$f$ be a normalized Belyi map of degree $d \geq 3$.
For any $c \in \mathbb{P}^1(\mathbb{Q})$, we may construct the
\emph{dynamical sequence} $\{f^i(c)\}_{i \geq 0}$, where $f^0(c) = c$,
and study the primes dividing at least one non-zero finite term of the sequence. (Note that since normalized Belyi maps fix $0, 1, \infty$, the only interesting dynamical sequences are those avoiding these three points.)

\begin{introtheorem}[Corollary \ref{cor:seq}(2)]\label{intro3}
Let $K$ be the splitting field of $f$ and consider the non-zero preimages of zero under $f$, denoted $c_j \in K$. Suppose that for each $c_j$ the analogues of Conditions~\ref{conditions} are satisfied for distinct $K$-primes $\mathfrak{p}_j, \mathfrak{q}_{1,j}, \mathfrak{q}_{2,j}, \mathfrak{q}_{3,j}$.  
Define a dynamical sequence $\{b_i\}_{i
  \geq 0}$ by $b_0 \in \mathbb{Q}\setminus \{0,1\}$ and $b_i = f(b_{i-1})$ for $i \geq 1$. Then
the set of prime divisors of the entries of this sequence has natural density zero.
\end{introtheorem}

\noindent {\bfseries Outline of the paper.}  In Section
\ref{sec:belyi}, we introduce normalized (dynamical) Belyi maps and
formulate known facts on their Galois groups. We also recall some
results from \cite{WINE2} on reduction of normalized Belyi maps.  In
Section \ref{sec:arbo} we consider the automorphism group
$\mathrm{Aut}(T_\infty)$ of an infinite regular tree. We prove our
first main result, Theorem \ref{thm:G2alt}, which describes the groups
$G_{n, \QQb}$  as subgroups of $\mathrm{Aut}(T_n)$, and prove Corollary~\ref{cor:condsgn}, which compares $G_{n, \QQ}$ with $G_{n,\QQb}$.

Section \ref{sec:spec} treats the specializations $G_{n,a}$. Our
second main result, Theorem \ref{G2aG2}, proves sufficient conditions
guaranteeing that $G_{n,a} \simeq G_{n, \QQb}$ for all $n \geq 1$.  In
Section \ref{sec:app}, we show in Theorem \ref{EnfixEn} that  
the proportion of elements of $G_n$ (and hence of
$G_{n,a}$, when Conditions \ref{conditions} are satisfied) that fix a
leaf tends to zero as $n$ tends to infinity. This is applied to derive
consequences for prime divisors of dynamical sequences in Corollary
\ref{cor:seq}.

\bigskip\noindent {\bfseries Acknowledgments.}  This project started
at the American Institute of Mathematics as part of an AIM SQuaRE on
dynamical Belyi maps.  We thank AIM for providing a stimulating
environment, and the other participants Jacqueline Anderson, Neslihan
Girgin, and Michelle Manes for many useful discussions. Our previous
paper \cite{WINE2}, which evolved from work done at  the
conference \textit{Women in Numbers Europe 2},  is a collaboration with them.  
We are grateful to Alexander Hulpke for his help with some group-theoretic
computations.  We gratefully acknowledge the support of the
Mathematisches Forschungsinstitut Oberwolfach, where part of the work
was carried out. The second and third authors were supported by an AMS -- Simons Travel Grant. 
We thank the anonymous referees for their careful reading and their helpful suggestions which greatly improved the exposition.

\section{Dynamical Belyi maps}\label{sec:belyi}

\subsection{Normalized Belyi maps of type $(d; e_1, e_2, e_3)$}\label{sec:Belyidef}

In this section we introduce the class of Belyi maps that we will study in this paper. Recall that a \textit{Belyi map} is a finite cover $f: X \to \mathbb{P}^1$ of smooth projective curves over $\mathbb{C}$ that is branched exactly over $x_1=0$, $x_2=1$, and $x_3=\infty$. 
A Belyi map has genus zero if $X$ has genus $g(X)=0$. 

\begin{remark}
A \emph{dynamical Belyi map} is a genus-$0$ Belyi map $f$ such
that $f(\{0,1,\infty\}) \subseteq \{0,1,\infty\}$.  This notion was
introduced in \cite{Zvonkin}.  In this case the $n$th iterate of $f$,
which we denote by $f^n$, is also a Belyi map of genus zero. All dynamical Belyi
maps are \emph{post-critically finite}; a map $f : \mathbb{P}^1
\to \mathbb{P}^1$ is post-critically finite (PCF) if each of
its ramification points has finite forward orbit.
\end{remark}

\begin{definition}\label{def:type}\mbox{}
\begin{enumerate}
\item A Belyi map $f$ is called \emph{single cycle} if there is a
  unique ramification point over each of the three branch points.
\item A single-cycle genus-$0$ Belyi map is called \emph{normalized}
  if its ramification points are $0$, $1$, $\infty$, and moreover
  $f(0) = 0$, $f(1) = 1$, and $f(\infty) = \infty$.
\item The \emph{(combinatorial) type} of a single-cycle genus-$0$
  Belyi map $f$ is the tuple $(d; e_1, e_2, e_3)$, where $d=\deg(f)$
  and $e_i$ is the ramification index of the unique ramification point
  above $x_i$.
\end{enumerate}
\end{definition}

Note that normalized (single-cycle genus-$0$) Belyi maps $f$ are
dynamical Belyi maps, and hence PCF maps. They even satisfy the stronger condition of
being \emph{conservative}, which means that they are PCF maps such
that each of their ramification points is a fixed point.

The Riemann--Hurwitz formula implies that the type of a single-cycle
genus-0 Belyi map of type $(d; e_1, e_2, e_3)$   satisfies
\begin{equation}\label{eq:genus0}
2d+1=e_1+e_2+e_3.
\end{equation}
A Belyi map of type $(d; e_1, e_2, e_3)$ is automatically single cycle. All Belyi maps considered in this paper are assumed to be normalized (single-cycle genus-$0$ dynamical) Belyi maps from now on.

To exclude trivial cases we always assume that all normalized Belyi
maps of type $(d; e_1, e_2, e_3)$ have exactly three ramification
points, i.e.,$e_i>1$ for all $1 \leq i \leq 3$.  For simplicity we will moreover
always assume that $e_1\leq e_2\leq e_3$.  Permuting the $e_i$
corresponds to changing coordinates on both projective lines
simultaneously, therefore these inequalities pose no restriction.  An
\textit{abstract type} is a tuple $(d; e_1, e_2, e_3)$ such that
\[
2\leq e_1\leq e_2\leq e_3\leq d
\]
 and such that (\ref{eq:genus0}) holds.  

The following result is classical; a proof can be found in
\cite[Proposition 1]{WINE2}. 
Note that the normalization condition completely fixes the coordinates on both projective lines.

\begin{proposition}\label{prop:rigid}
For each abstract type $\underline{C}:=(d; e_1, e_2, e_3)$ there is a
unique normalized Belyi map $f$ of type $\underline{C}$. Moreover,
this Belyi map is defined over $\QQ$.
\end{proposition}

Proposition \ref{prop:rigid} implies that a normalized Belyi map $f$ is 
completely determined by its type.  All proofs in this paper only
depend on the type of~$f$ and do not use the explicit equation for
$f$.  However, it is not too difficult to explicitly calculate the
normalized Belyi map of a given type, as the following result
illustrates. 

\begin{proposition}\label{prop:Belyifam} Let $d \geq 3$ and $k \geq 1$.
\begin{enumerate}
\item The unique normalized Belyi map of type $(d; d-k, k+1, d)$ is
\[
f(x)=cx^{d-k} (a_0x^k+\cdots +a_{k-1}x+a_k),
\]
where 
\[
a_i:= \frac{(-1)^{k-i}}{(d-i)} \binom{k}{i} \text{ and } 
c=\frac{1}{k!} \prod_{ j=0}^{k}(d-j).
\] 
\item  
The unique normalized Belyi map of  type $(d; d-k, 2k+1, d-k)$ is
\[
f(x)=x^{d-k}\left( \frac{a_0x^k-a_1x^{k-1}+\cdots +(-1)^ka_k}{(-1)^k
  a_k x^k +\cdots -a_1x+ a_0}\right),
\]
where
\[ 
a_i:=\binom{k}{i}\prod_{k+i+1\leq j \leq 2k} (d-j) \prod_{0\leq j \leq i-1}(d-j) =
k!\binom{d}{i}\binom{d-k-i-1}{k-i}.
\] 
\end{enumerate}
\end{proposition}

\begin{proof} This is \cite[Propositions 2 and
  3]{WINE2}. \end{proof}

Let $f$ be a normalized Belyi map of type $(d; e_1, e_2,
e_3)$. Associated with $f$ is a \textit{generating system} $(g_1, g_2,
g_3)$, which describes the quotient of the topological fundamental
group $\pi_1(\PP^1\setminus\{0,1,\infty\}, \ast)$ corresponding to
$f$.  In our situation, a generating system consists of three
permutations $g_i\in S_d$ for $1 \leq i \leq 3$, where $g_i$ is a single cycle of length
$e_i$, that satisfy the relation $g_1g_2g_3=1$.  This motivates the
terminology \textit{single cycle}. The generating system for $f$ is
unique up to simultaneous conjugacy by elements of $S_d$.  The group
$G(f):=\langle g_1, g_2, g_3\rangle$ is the Galois group of the Galois
closure of the cover $f:\PP^1_{\overline{\QQ}}\to \PP^1_{\overline{\QQ}}$, 
i.e., the Galois group of the splitting field of $f(x)-t$ over $\overline{\QQ}(t)$. 
Liu--Osserman (\cite[Lemma
  2.1]{liuosserman} ) show that the triple $(g_1, g_2, g_3)$ is
\textit{weakly rigid} in the sense that it is unique up to uniform
conjugacy in $S_d$. However, in the case that $G(f)\subsetneq S_d$ the
generating system is not unique up to uniform conjugacy in $G(f)$,
i.e., the triple is not rigid. In this case $G(f)$ is not necessarily
the Galois group of the cover $f:\PP^1_{\QQ}\to \PP^1_{\QQ}$, even though the map $f$ is defined
over $\QQ$. We
discuss this phenomenon in more detail in Section \ref{sec:descent}.
We refer to \cite[Chapter 2]{volklein} for a general introduction to
this topic.

\begin{lemma}\label{lem:gensys}
Let $f$ be a normalized Belyi map of type $\underline{C}=(d; e_1, e_2, e_3)$.   
\begin{enumerate}
\item Assume  $\underline{C} \neq (6; 4,4,5)$. Then the Galois group $G(f)$ of the Galois closure of $f:\PP^1_{\overline{\QQ}}\to \PP^1_{\overline{\QQ}}$ is isomorphic to $S_d$ if at least one $e_i$ is even, and to $A_d$ otherwise. 
\item Define
\begin{equation}\label{eq:gensysn=1}
\begin{split}
g_1 & =  (d,d-1,\ldots,e_3,1,2, \ldots, d-e_2-1,d-e_2),\\
 g_2 & =  (d-e_2+1, d-e_2+2, \ldots, d-1,d), \\
 g_3 & =  (e_3,e_3-1,\ldots, 2,1).
 \end{split}
 \end{equation}
Then $(g_1, g_2, g_3)$ is a generating system for $f$.
\end{enumerate}
\end{lemma}

\begin{proof}
The first statement is proved in \cite[Theorem~5.3]{liuosserman}. The
second statement is \cite[Lemma~2.1]{liuosserman}.
\end{proof}

\subsection{Reduction of normalized Belyi maps}\label{sec:Belyired}

In this section we recall from \cite[Section 4]{WINE2} the definition of and some results on the reduction of normalized Belyi maps.\\

We identify the cover $f$ with the rational function that defines it. Without loss of generality we may assume the following holds:
\begin{enumerate}
\item We may write $f(x) = x^{e_1}f_1(x)/f_2(x)$,  with $f_1, f_2\in \ZZ[x]$ such that $f_1(0) \neq 0$ and $f_2(0) \neq 0$;
\item The polynomials $f_1, f_2$ satisfy $\gcd(f_1, f_2)=1$ and have coprime leading coefficients;
\item The greatest common divisor of all coefficients of $f$ is $1$.
\end{enumerate}

\begin{definition}\label{def:Belyired}
Let $f$ be a normalized Belyi map of type $(d; e_1, e_2, e_3)$ and let $p$ be a prime. 
The \emph{reduction} $\overline{f}$ of $f$ at $p$ is defined as $\overline{f}=x^{e_1}\overline{f}_1/\overline{f}_2$,
where $\overline{f}_i$ is the reduction of $f_i$ modulo $p$ (as a polynomial with coefficients in $\ZZ$).
\end{definition}

Definition \ref{def:Belyired} is a simplification of the definition in
\cite[Section 4]{WINE2}, taking into account the result of
\cite[Proposition 4]{WINE2}.

The following proposition lists some properties of the reduction of a
normalized Belyi map.  The proof uses the assumption that $f$ is a normalized
Belyi map of type $\underline{C}$ in an essential way, see
\cite[Example 4]{WINE2}.

\begin{proposition}\label{prop:Belyired}
Let $f$ be a normalized Belyi map of type $\underline{C}=(d; e_1, e_2, e_3)$ and let $p$ be a prime.
\begin{enumerate}
\item The rational function $\overline{f}\in \FF_p(x)$ is non-constant.
\item We have $\overline{f}(0)=0, \overline{f}(1)=1$, and $\overline{f}(\infty)=\infty$.
\end{enumerate}
\end{proposition}

\begin{proof}
The statement is not stated explicitly in \cite{WINE2} in this form,
but it follows immediately from the proof of \cite[Proposition
  4]{WINE2} and \cite[Remark 6]{WINE2}.
\end{proof}

\begin{definition}\label{def:goodred}\mbox{}
\begin{enumerate}
\item A normalized Belyi map $f$ has \emph{bad reduction} at a
  prime $p$ if $\overline{f}$ has strictly smaller degree than
  $f$. Otherwise, $f$ has \emph{good reduction} at $p$.
\item A normalized Belyi map $f$ has \emph{good monomial reduction
  at $p$} if it has good reduction at $p$ and its reduction is
  $\overline{f}(x) = x^{\deg(f)}$.
\item A normalized Belyi map $f$ has \emph{good
  separable reduction at $p$} if it has good reduction at $p$ and its
  reduction is a separable rational map. A rational map in $\FF_p(x)$ is separable if and only if it is not contained in $\overline{\FF}_p(x^p)$, if and only if it defines a separable map $\mathbb{P}^1_{\mathbb{F}_p} \to \mathbb{P}^1_{\mathbb{F}_p}$.
\end{enumerate}
\end{definition}

If a normalized Belyi map $f$ has good separable reduction at $p$ then
its reduction $\overline{f}$ is also a normalized Belyi map and
$\overline{f}$ has the same type as $f$. This follows from the proof
of \cite[Proposition 5]{WINE2}. Good monomial reduction is a special
case of good inseparable reduction. In Proposition
\ref{prop:Eisenstein} we use it to characterize the irreducibility of
specializations of $f$ and its iterates. 

The following result allows us to prescribe the reduction of a normalized Belyi map by purely combinatorial conditions on its ramification indices.

\begin{proposition}\label{prop:Belyiredcrit}
Let $f$ be a normalized Belyi map of combinatorial type $(d; e_1, e_2, e_3)$.
\begin{enumerate}
\item Assume that $p>d$. Then $f$ has good separable reduction at $p$.
\item Write $d = p^nd'$ where $p \nmid d'$. Then $f$ has good monomial
  reduction at $p$ if and only if $e_2 \leq p^n$.
\end{enumerate}
\end{proposition}

\begin{proof}
Statement (1) follows from \cite[Corollary 2]{WINE2}.  Statement~(2)
is \cite[Theorem 1]{WINE2}.
\end{proof}

\section{Arboreal Galois groups}\label{sec:arbo}
\subsection{Automorphisms of the $d$-ary regular tree}\label{sec:tree}

Let $T_n$ be the regular $d$-ary rooted tree of level $n$ (cf.~Figure \ref{fig1}).
 \begin{figure}[ht]
\begin{tikzpicture}[node distance=2cm]
 \node(zero) at (0,5) {$\circ$};
 \node (circ1) at (-1,4) {$\circ$};
 \node (circ2) at (0,4) {$\circ$};
 \node (circ3) at (1,4) {$\circ$};
 \node (prevcircle1) at (-2.5,2) {$\circ$};
 \node (prevcircle2) at (0,2) {$\circ$};
 \node (prevcircle3) at (2.5,2) {$\circ$};
 \node (v1) at (-3,1) {$\circ$};
 \node (v1circ1) at (-2.5,1) {$\circ$};
 \node (v1circ2) at (-2,1) {$\circ$};
 \node (vk) at (0,1) {$\circ$};
 \node (vkcirc1) at (-0.5,1) {$\circ$};
 \node (vkcirc2) at (0.5,1) {$\circ$};
 \node (vN) at (3,1) {$\circ$};
 \node (vNcirc1) at (2.5,1) {$\circ$};
 \node (vNcirc2) at (2,1) {$\circ$};
 \draw[dotted] (circ1)--(prevcircle1);
 \draw[dotted] (circ2)--(prevcircle2);
 \draw[dotted] (circ3)--(prevcircle3);
 \draw (prevcircle1) -- (v1);
 \draw (prevcircle1) -- (v1circ1);
 \draw (prevcircle1) -- (v1circ2);
 \draw (prevcircle2) -- (vk);
 \draw (prevcircle2) -- (vkcirc1);
 \draw (prevcircle2) -- (vkcirc2);
 \draw (prevcircle3) -- (vN);
 \draw (prevcircle3) -- (vNcirc1);
 \draw (prevcircle3) -- (vNcirc2);
 \draw (v1) -- (prevcircle1);
 \draw (circ1) -- (zero);
 \draw (circ2) -- (zero);
 \draw (circ3) -- (zero);
\end{tikzpicture}
\caption{The regular $3$-ary tree}\label{fig1}
\end{figure}
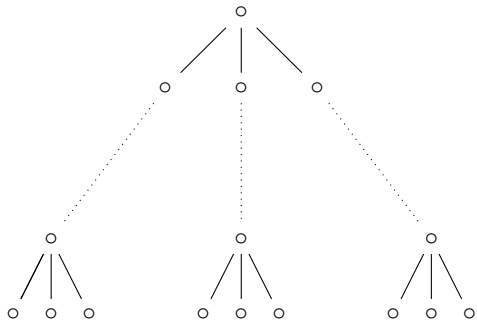

We label the vertices of the tree as follows: the root of the tree
corresponds to the level $0$, and has the empty label $()$.  The
vertices at level $i$ are labeled as $(\ell_1, \ldots, \ell_i)$ with
$\ell_j\in \{1, \ldots, d\}$.  Here $ (\ell_1, \ldots, \ell_{i-1})$ is
the unique vertex at level $i-1$ which is connected to $(\ell_1,
\ldots, \ell_i)$ by an edge.  For the vertices of $T_n$ at level $n$,
also called the leaves, we additionally use the numbering
\begin{equation}\label{eq:leafnumber}
1+\sum_{k=1}^n(\ell_k-1) d^{k-1}\in \{1,2,\ldots, d^n\} \qquad \text{ instead of } (\ell_1,
\ldots, \ell_n).
\end{equation}
Since $\Aut(T_n)$ acts faithfully on the leaves of the tree $T_n$, the
choice of the numbering induces an injective group homomorphism
\begin{equation}\label{eq:iota}
\iota_n:\Aut(T_n)\hookrightarrow S_{d^n}.
\end{equation}
In this paper we use the convention that permutations act from the right. 
 
Rather than considering $\Aut(T_n)$ as a subgroup of $S_{d^n}$ it is
more convenient for our purposes to view $\Aut(T_n)$ as a subgroup of
the $n$-fold iterated wreath product of $S_d$ by itself. The structure
of $\Aut(T_n)$ as $n$-fold iterated wreath product is 
\begin{equation}\label{eq:wrauto}
\Aut(T_n)\simeq \Aut(T_{n-1})\wr \Aut(T_1) \text{ for } n \geq 2.
\end{equation} 
This isomorphism is induced by the decomposition of $T_n$ as the
subtree $T_1$ (consisting of the levels $0$ and $1$), and $d$ copies
of $T_{n-1}$ each consisting of the complete subtree of $T_n$ with
root $(j)$ for $j\in \{1, \ldots, d\}$. We remark that
$\Aut(T_1)\simeq S_d$, but that the iterated wreath product
$\Aut(T_n)$ is a strict subgroup of $S_{d^n}$ for $n\geq 2$.

Equation  \eqref{eq:wrauto}  allows us to write the elements of
$\Aut(T_n)$ as tuples
\begin{equation}\label{eq:wrelts}
(\boldsymbol{\sigma}, \tau)=((\sigma_1, \ldots, \sigma_d), \tau) \text{ with }\sigma_j\in \Aut(T_{n-1})\text{ and }\tau\in \Aut(T_1)\simeq S_d.
\end{equation}
Using this identification, an element of $\Aut(T_n)$ acts on vertices of $T_n$ as
\begin{equation}\label{eq:wraction} 
(\ell_1, \ldots, \ell_n)\cdot ((\sigma_1, \ldots,\sigma_d), \tau) =
 ( (\ell_1)\tau, (\ell_2 \ldots, \ell_n)\sigma_{\ell_1} ).
\end{equation}
In other words, $\tau$ permutes the $d$ complete subtrees isomorphic
to $T_{n-1}$, and $\sigma_j$ acts on the complete subtree with root
$(j)$.

Let $(\ell_1, \ldots, \ell_m)$ be a vertex of $T_n$ at level $m\leq
n-1$ and let $i=1+\sum_{k=1}^m (\ell_k-1)d^{k-1}$.  We denote the
complete subtree of $T_n$ with root $(\ell_1, \ldots, \ell_m)$ by
\begin{equation}\label{eq:subtree}
T_n(\ell_1, \ldots, \ell_m)=T_n^i.
\end{equation}
Correspondingly, we denote the group of permutations that only permute
the leaves of $T_n^i$ and fix all other leaves of $T_n$ by
\begin{equation}\label{eq:symmsubtree}
\mathcal{S}(T_n^i)=\{ (\boldsymbol{\sigma}, \tau)\in \Aut(T_n)\mid  (\boldsymbol{\sigma}, \tau) \text{ acts trivially outside  }T_n^i).
\end{equation}
We may view 
  $\mathcal{S}(T_n^i)$ as a subgroup of $S_{d^n}$ via $\iota_n$.

For future reference, we note that
\begin{equation}\label{eq:wrconj}
\begin{split}
((\sigma'_1, \ldots,\sigma'_d), \tau')\cdot ((\sigma_1,
  \ldots,\sigma_d), \tau) &= ( (\sigma'_1\sigma_{(1)\tau'}, \ldots,
  \sigma'_d\sigma_{(d)\tau'}), \tau'\tau),\\ 
  ((\sigma_1,  \ldots,\sigma_d), \tau)^{-1}&=((\sigma_{(1)\tau^{-1}}^{-1},\ldots,\sigma_{(d)\tau^{-1}}^{-1}), \tau^{-1})
 ),\\ 
  ((\sigma_1, \ldots,\sigma_d), \tau)^{-1} ((\xi_1,
  \ldots,\xi_d),-) ((\sigma_1, \ldots,\sigma_d), \tau)&=((\sigma^{-1}_{(1)\tau^{-1}}\xi_{(1)\tau^{-1}}\sigma_{(1)\tau^{-1}},
  \ldots), -).
 \end{split}
\end{equation}
Here $-$ denotes the trivial permutation.

For every $m\leq n$ we write $\pi_m$ for the natural projection
\[
\pi_m:\Aut(T_n)\to \Aut(T_m),
\]
which corresponds to restricting the action of an element of
$\Aut(T_n)$ to the subtree $T_m$ consisting of the levels $0, 1, \ldots, m$.

\begin{definition}\label{def:sgn2}\mbox{}
\begin{enumerate}
\item  Define $\sgn_2:\Aut(T_2)\to \{\pm 1\}$ by
setting
 \begin{equation}\label{eq:wrsgn}
 \sgn_2(((\sigma_1, \ldots,\sigma_d), \tau)) = \sgn(\tau) \prod_{i=1}^d
 \sgn(\sigma_i).
 \end{equation}
Here $\sgn$ is the usual sign on $\Aut(T_1)$ via the identification $
\Aut(T_1)\simeq S_d$ induced by the choice of labeling of the vertices.
\item For $n> 2$ we define
\begin{equation}\label{eq:sgn2}
\sgn_2:=\sgn_2\circ \pi_2:\Aut(T_n)\to \{\pm 1\}.
\end{equation}
\end{enumerate}
We call $\sgn_2$  the \emph{wreath-product sign}.
\end{definition}

 Using  \eqref{eq:wrconj}, one may check that $\sgn_2$ is a group
 homomorphism.  We define a subgroup $E_n$ of $\Aut(T_n)$ using the
 wreath-product sign. We show in Corollary \ref{cor:GnEn} below that
 the Galois group of $f^n$ is a subgroup of $E_n$.

\begin{definition}\label{def:En}
Define the subgroup $E_n\subseteq \Aut(T_n)$ by
\[
E_n=\begin{cases}  \Aut(T_1)&\text{ if }n=1,\\ (E_{n-1}\wr
E_1)\cap \ker(\sgn_2)\subseteq \Aut(T_{n-1})\wr\Aut(T_1) = \Aut(T_n)
&\text{ otherwise}.
\end{cases}
\]
\end{definition}

\begin{lemma}\label{lem:En}
\begin{enumerate}
\item For all $d\geq 2$ and $n\geq 2$ we have 
\[
[\Aut(T_n):E_n]=2^{d^{n-2}+d^{n-3}+\cdots+d+1}.
\]
\item Assume that $d$ is odd. Then the wreath-product sign on $\Aut(T_n)$ agrees with the restriction of the usual sign on
  $S_{d^2}$ via the embedding $\iota_n$  from  \eqref{eq:iota}:
\[
\sgn_2:\Aut(T_2)\to\{\pm 1\}, \qquad \sgn_2((\boldsymbol{\sigma},
\tau))\mapsto \sgn\circ \iota_2((\boldsymbol{\sigma}, \tau)).
\]
\end{enumerate}
\end{lemma}

\begin{proof}
The definition of $E_n$
(Definition \ref{def:En}) implies that 
\[
|E_n|=|E_{n-1}|^d\cdot |E_1|/2.
\]
Since $E_1=\Aut(T_1)$, the definition of the wreath-product sign
 \eqref{eq:sgn2}  implies that
\[
[\Aut(T_n):E_n]=2\left([\Aut(T_{n-1}):E_{n-1}]\right)^d.
\]
Statement (1) follows from this by induction.

To prove Statement (2), we first note that for any $(\boldsymbol{\sigma}, \tau) \in \Aut(T_2)$ we may write
\[
(\boldsymbol{\sigma}, \tau)=(\boldsymbol{\sigma},
-)\circ ((-), \tau).
\]
Since $\sgn_2$ is a group homomorphism, it suffices to treat the
elements on the right hand side separately.  

From  \eqref{eq:wraction}  it follows that $((-), \tau)$ acts as the
product of $d$ 
disjoint permutations of the same cycle type as
$\tau$. We conclude that the sign in $S_{d^2}$ of
the image of $((-), \tau)$ under $\iota_2$ is
$\sgn(\tau)^{d}$. This equals the wreath-product sign
$\sgn_2(((-), \tau))=\sgn(\tau)$  if $d$ is odd.

Now let $(\boldsymbol{\sigma},-)=((\sigma_1, \ldots, \sigma_d), -)\in
\Aut(T_2)$ with $\sigma_i\in E_{1}=\Aut(T_1)$. This element acts as
the product $\prod_i \sigma_i$, where each $\sigma_i$ acts on the
subtree of $T_2^i$, i.e., as element of $\mathcal{S}(T_2^i)$. (See
 \eqref{eq:symmsubtree}.)  We may therefore identify
$(\boldsymbol{\sigma},-)$ with its image under $\iota_2$ in $S_{d^2}$.
We
conclude that
\[
\sgn_2((\boldsymbol{\sigma}, -))=\prod_j\sgn(\sigma_j)=\sgn(\iota_2(\sigma_1,
\ldots, \sigma_d)).
\]
This finishes the proof of Statement (2).
\end{proof}

Note that it follows from the proof of Lemma \ref{lem:En}(2)
that for $d$ even the wreath-product sign  \eqref{eq:wrsgn}  is not
compatible with the natural sign on~$S_{d^2}$.

\subsection{A generating system for $f^n$}\label{sec:gensys}

Let $f$ be a normalized Belyi map. 
Recall that $f^n=f\circ \cdots \circ f$ denotes the $n$th iterate of
$f$, which is again a normalized Belyi map.
We start by defining the Galois groups that are the central object of
study in this paper. We then determine a generating system for $f^n$
in terms of a generating system for $f$.

Let $K$ be a field of definition of $f$. Since $f$ is a Belyi map, we
may assume that $K$ is a number field. Moreover, since $f$ is normalized, it
follows from Proposition \ref{prop:rigid} that we may take
$K=\QQ$. 
Write $F_0:=K(t)$ for the function field of $\PP^1_K$. Since
$K$ is a field of definition for $f^n$ for all $n$, the field $K$ is
integrally closed in the extension of function fields corresponding
to the map $f^n:\PP^1_K\to\PP^1_K$, which we denote by $F_n/F_0$.  We
choose a normal closure $M_n/F_0$ of $F_n/F_0$ such that $M_n$
contains $M_{n-1}$ for any $n \geq 1$.

The extension of function fields $\left(F_n\otimes_K\QQb\right)/\QQb(t)$
corresponds to $ f^n:\PP^1_\QQb\to\PP^1_\QQb$ considered as map over
the algebraic closure $\QQb$ of the number field $K$. Note that
$\left(M_n\otimes_K \QQb\right)/\QQb(t)$ is a normal closure of $\left(F_n \otimes_K
\overline{\QQ}\right)/\overline{\QQ}(t)$.  

\begin{definition}\label{def:Gn}
For arbitrary $n \geq 1$ we define
\[
G_{n, \QQ} =\Gal(M_n/F_0), \qquad G_{n, \QQb} =\Gal(\left(M_n\otimes_K\QQb\right)/\QQb(t)).
\]
\end{definition}

Note that $G_{1,\QQb} = G(f)$ as defined in Section
\ref{sec:Belyidef}.  

It follows from the definitions  that 
\begin{equation}\label{eq:GnQQ}
G_{n, \QQb}\subseteq G_{n, \QQ}.
\end{equation}

In general it is not true that $G_{n, \QQb}= G_{n, \QQ}$, see Remark
\ref{rem:disc}. In the case that we have equality in~\eqref{eq:GnQQ} we say that \emph{the Galois extension
  $\left(M_n\otimes_K\QQb\right)/\QQb(t)$ descends to} $\QQ(t)$ or in short that
the group $G_{n\,\QQb}$ \emph{descends} to $\QQ$.

Our convention that the normal closure $M_n$ contains $M_m$ for all
$m\leq n$ implies that $G_{n, \QQ}$, and hence $G_{n,\QQb}$, naturally
has the structure of a wreath product. Identifying the sheets of $f^n$
above a chosen base point in $\PP^1(\QQ)\setminus\{0,1,\infty\}$ with
the leaves of the tree $T_n$ yields an inclusion
\begin{equation}\label{eq:Gnaswr}
G_{n,\QQ}\hookrightarrow \Aut(T_n)\simeq \Aut(T_{n-1})\wr \Aut(T_1).
\end{equation}
In the rest of the paper we fix this inclusion.

Our next goal is to determine a generating system $(g_{1,n}, g_{2,n},
g_{3,n})$ for $f^n$ for all $n$ in terms of a fixed generating system
$(g_{1,1}, g_{2,1}, g_{3,1})$ for $f$. Since $G_{n,\QQb}=\langle
g_{1,n}, g_{2,n}\rangle$, we can use this to determine the group
$G_{n,\QQb}$. We refer to Theorem \ref{thm:G2alt} for the precise
result.

\begin{remark}\label{rem:ramfn}
 By the single-cycle condition, the fiber above $0$ in $f^n$
 contains a unique point with ramification index $e_1^n$. (This is the
 point $x = 0$.)
 Additionally, for each $0\leq i\leq n-1$ 
 there are exactly $(d-e_1)d^{n-1-i}$ ramification points with
 ramification index $(e_1)^i$.  (These are exactly the points in the
 inverse image $f^{-n+i}(0)$ which are not in $f^{-n-1+i}(0)$.)
 The analogous statements with $0$ replaced by $1$ or $\infty$ and
 $e_1$ by $e_2$ or $e_3$, respectively, also hold.  This description
 determines the cycle type of the group elements of a generating
 system of $f^n$ considered as elements of $S_{d^n}$. For our purposes
 we need more precise information, which the following proposition
 supplies.
\end{remark}

A number is said to be \emph{in the support} of a permutation if it appears in one of the cycles of the permutation, i.e., if it is not fixed by the permutation.

\begin{proposition}\label{prop:gensys}
Let $(g_{1,1}, g_{2,1}, g_{3,1})$ be a generating system of $f$ such that $1$ is in the support of~$g_{1,1}$. 
For any $n \geq 2$, inductively define
\[
\begin{split}
g_{1,n}&=((g_{1,n-1}, -, \ldots,-), g_{1,1});\\
g_{2,n}&=((-,\ldots, -, g_{2,n-1}, -, \ldots,-), g_{2,1}), \text{ where $g_{2,n-1}$ is in position } (1)g_{1,1};\\
g_{3,n} &=( (-,\ldots, -, g_{3,n-1}, -, \ldots,-), g_{3,1}), \text{ where $g_{3,n-1}$ is in position } (1)g_{1,1}g_{2,1}.
\end{split}
\]
Here again $-$ denotes the trivial permutation. 
Then $(g_{1,n}, g_{2,n}, g_{3,n})$ is a generating system for~$f^n$.
\end{proposition}

\begin{proof}
This follows by considering the image of the sheets of the Belyi map
under $f^n=f\circ f^{n-1}$, using the notation introduced in
\eqref{eq:wrelts}.
\end{proof}

A generating system for $f$ was given in Lemma \ref{lem:gensys}(2).

\begin{corollary}\label{cor:GnEn}
For all $n\geq 2$ we have that
\[
G_{n,\QQb}\subseteq E_n.
\]
\end{corollary}

\begin{proof}
We have already seen that $G_{n, \QQb}\subseteq \Aut(T_n)$. Since
$G_{n,\QQb}$ is generated by the $g_{j,n}$, it suffices to check that
$g_{j,n}\in E_n$ for $j=1,2,3$.

The inductive definition of $g_{j,n}$ given in Proposition
\ref{prop:gensys} implies that $\pi_2(g_{j,n})=g_{j,2}$ and that $\sgn_2(g_{j,2})=1$ for
$j\in \{1,2,3\}$.  The statement for $n=2$ follows. The statement for
arbitrary $n\geq 2$ follows by induction from the definition of $E_n$
(Definition \ref{def:En}).
\end{proof}

\begin{definition}\label{def:NNi}
Define
\[
N_{n, \overline{\QQ}}=\ker(\pi_{n-1}:G_{n, \overline{\QQ}}\to G_{n-1, \overline{\QQ}})\subseteq G_{n, \overline{\QQ}}.
\]
\end{definition}

Note that $(\boldsymbol{\sigma}, \tau)\in G_{n, \QQb}$ is contained in
$ N_{n, \overline{\QQ}}$ if and only if $(\boldsymbol{\sigma}, \tau)$ fixes all vertices
of the tree $T_n$ at the levels $1, \ldots, n-1$.  For any $1\leq
i=\ 1+\sum_{k=1}^{n-1}(\ell_k-1)d^{k-1}\leq d^{n-1}$, the subtree
$T_n^i$ with root $(\ell_1, \ldots, \ell_{n-1})$ contains exactly $d$
leaves.

\begin{definition}\label{def:Ni}
For $1\leq i\leq d^{n-1}$, we define 
\[
N_{n, \overline{\QQ}}^i:=N_{n, \overline{\QQ}}\cap \mathcal{S}(T_n^i).
\]
\end{definition}

The group $N_{n, \overline{\QQ}}^i$ is naturally a permutation group
on the $d$ letters $(\ell_1, \ldots,\ell_{n-1}, s)$ for $s=1, \ldots,
d$; recall from \eqref{eq:leafnumber} that
$i=\ 1+\sum_{k=1}^{n-1}(\ell_k-1)d^{k-1}$.  We therefore obtain an
identification
\[
N_{n, \overline{\QQ}}^i\subseteq N_{n, \overline{\QQ}}^1\oplus \cdots
\oplus N_{n, \overline{\QQ}}^{d^{n-1}} = N_{n, \overline{\QQ}}
\subseteq (S_d)^{d^{n-1}}.
 \]
With this identification, we may write
\begin{equation}\label{eq:rhodef}
N_{n, \overline{\QQ}}=\{(\rho^1, \ldots, \rho^{d^{n-1}})\mid \rho^i\in
\mathcal{S}(T_n^i) \text{ for all } 1 \leq i \leq d^{n-1} \}\subseteq (S_d)^{d^{n-1}}.
\end{equation}

\begin{lemma}\label{lem:Ni}
The group $N_{n,\overline{\mathbb{Q}}}^i$ is a subgroup of $A_d$ for all $i \in \{1,\ldots, d^{n-1}\}$.
\end{lemma}
\begin{proof}
We prove this for $i=1$; the proof for other $i$ is identical. 
First let $n=2$. If $\sigma=(\rho^1,-,\ldots,-)$ is an element of $N_{2,\overline{\mathbb{Q}}}^i$, then $\sgn_2(\sigma)=\sgn(\rho^1)=1$ by Corollary~\ref{cor:GnEn}. 
 Now let $n \geq 2$. If $\sigma=(\rho^1,\ldots,\rho^{d^{n-1}})$ is in $N_{n,\overline{\mathbb{Q}}}$, then $(\rho^1,\ldots,\rho^{d^{n-2}})$ is an element of $N_{n-1,\overline{\mathbb{Q}}}$, since $G_n$ can be identified with a subgroup of $G_{n-1} \wr G_1$ for any $n\geq 2$. Hence the proof follows by induction on ~$n$.
\end{proof}

In the rest of this section, we fix a normalized Belyi map $f$ of type $\underline{C}=(d; e_1, e_2,
e_3)$. We use Proposition \ref{prop:gensys} to construct suitable
elements in $N_{n, \overline{\QQ}}$. This is a first step towards determining $G_{n,\QQb}$ in Theorem \ref{thm:G2alt}.

\begin{lemma}\label{alphas}
Let $f$ be a normalized Belyi map of type $(d; e_1, e_2, e_3)$ and let $(g_{1,1}, g_{2,1}, g_{3,1})$ be a generating system of $f$.  For any
$n\geq 2$ and $j\in \{1,2,3\}$ we define
\[
\alpha_{j,n}:=(g_{j,n})^{e_j^{n-1}}.
\]
\begin{enumerate}
\item
Then
\[
\alpha_{j,n}\in N_{n, \overline{\QQ}}.
\]
\item
Write $\rho_j^i$ for the component of $\alpha_{j,n}$ in $\mathcal{S}(T_n^i)$
 using the notation from \eqref{eq:rhodef}.
Then
\[
\rho_{j}^i=\begin{cases}
g_{j,1}& \text{ if }i=1+\sum_{k=1}^{n-1}(\ell_k-1)d^{k-1} \text{ with $\ell_k\in\supp(g_{j,1})$ for all $k$, }\\
- & \text{ otherwise}.
\end{cases}
\]
\item Conjugation by the elements of $G_{n, \QQb}$ acts transitively
  on $N_{n, \overline{\QQ}}^1, \ldots, N_{n, \overline{\QQ}}^{d^{n-1}}$ for all $n\geq 2$. 

\end{enumerate}
\end{lemma}

\begin{proof}
Let $i\in \{1,2,3\}$ and $n\geq 2$ arbitrary.  Using \eqref{eq:wrconj}
one computes that
\begin{equation}\label{eq:alpha}
\alpha_{j,n}=((\sigma_{1}, \ldots, \sigma_{d}), -), \qquad \text{ where }
\sigma_{i}=\begin{cases} \alpha_{i,n-1}& \text{ if }i\in \supp(g_{j,1}),\\
-&\text{ otherwise.}
\end{cases}
\end{equation}

Expression  \eqref{eq:alpha} for $\alpha_{j,n}$ implies that
$\alpha_{j,n}$ fixes all vertices on level one. If $\alpha_{j,n-1}\in
N_{n-1, \overline{\QQ}}$ then $\alpha_{j,n}$ fixes all vertices of the
tree on levels $2, \ldots,n-1$, as well. It follows that
$\alpha_{j,n}\in N_{n, \overline{\QQ}}$.  Statement~(1) of the
proposition is vacuous for $n=1$. The statement therefore follows by
induction on~$n$.

Moreover, $\alpha_{j,n}$ acts as $\alpha_{j,n-1}$ on the subtree $T_n(\ell_1)$ of
$T_n$ with root $(\ell_1)$ if $\ell_1\in \supp(g_{j,1})$ and acts trivially otherwise. Statement (2) therefore also follows by induction
on $n$.

Statement (3) follows by induction, as well, since $G_{1, \QQb}=G(f)$ is a
transitive subgroup of $S_d$.
\end{proof}

In the following proposition we exclude two types in small degree. For
$\underline{C}=(6; 4,4,5)$ we have that $G_{1, \QQb}$ is $S_5$
(embedded as a transitive group in $S_6$), hence is isomorphic to
neither $S_6$ nor $A_6$. For $\underline{C}=(4; 3,3,3)$ we have that
$G_{1, \QQb}\simeq A_4$. However in this case Proposition
\ref{prop:betas} fails. (In the case that $\underline{C}=(4; 3,3,3)$ the
group $N_{n, \QQb}^i$ is the Klein $4$-group.)

\begin{proposition}\label{prop:betas}
Let $f$ be a normalized Belyi map of type $\underline{C}$. Assume that $\underline{C}\notin\{ (6; 4,4,5), (4;3,3,3)\}$.  Then 
\[
N_{n, \overline{\QQ}}^i\simeq A_d
\]
for all $1 \leq i \leq d^{n-1}$ and all $n\geq 2$.
\end{proposition}

\begin{proof}
Lemma \ref{alphas}(3) states that $G_{n, \QQb}$ acts transitively on
the set of subgroups $ N_{n, \overline{\QQ}}^i$ for $i=1,\ldots
d^{n-1}$. Therefore it suffices to prove the proposition for a
specific value of $i$. We prove that $N_{n, \overline{\QQ}}^{i_0}$ is
a non-trivial normal subgroup of $G_{1, \overline{\QQ}}$ for some
value $i_0$, by showing it is normal and contains a non-trivial
element $\beta_n$.  Since $A_d$ is simple for $d \geq 5$ and $N_{n,
  \overline{\QQ}}^i \leq A_d$ for all $i$ by Lemma~\ref{lem:Ni}, it follows
that $ N_{n, \overline{\QQ}}^{i_0}=A_d$. In the remaining cases $d\in
\{3,4\}$ the statement can be shown by treating each type separately.

\bigskip
\textbf{Claim 1}: 
Let $i_0:=1+\sum_{k=1}^{n-1}(e_3-1)d^{k-1}$.
There exists a non-trivial element $\beta_n\in N_{n, \overline{\QQ}}^{i_0}$, hence $N_{n, \overline{\QQ}}^{i_0}$ is non-trivial. \\
Define
\begin{equation}\label{beta}
\beta_n:=[\alpha_{1,n}, [\alpha_{2,n},\alpha_{3,n}]] =\alpha_{1,n}[\alpha_{2,n},\alpha_{3,n}]\alpha_{1,n}^{-1}[\alpha_{2,n},\alpha_{3,n}]^{-1}, 
\end{equation} 
where $[\alpha_{2,n},
  \alpha_{3,n}]=\alpha_{2,n}\alpha_{3,n}\alpha_{2,n}^{-1}\alpha_{3,n}^{-1}$
is the commutator.
 Recall that
\[ 
\alpha_{j,n}=((\sigma_1,\ldots, \sigma_d), -) \in G_{n-1}\wr G_1,
\]  where $\sigma_i=\alpha_{j,n-1} \in G_{n-1}$ for $i$ in the support of $g_{j,1}$ and trivial otherwise. 
Since the generating system is weakly rigid, to prove the claim it
suffices to prove that the element $\beta_n$ is a non-trivial element
in $N_n^{i_0}$ for the generating system of Lemma
\ref{lem:gensys}(2). We use this generating system for the rest of
this proof. 

With this choice we have
\[
\supp(g_{1})\cap \supp(g_{2})\cap \supp(g_{3})=\{e_3\}.
\]
By induction we find that the component $\rho^i$ of $\beta_n$ in
$N_{n,\QQb}^i$ satisfies
\[
\rho^i=\begin{cases} [g_{1},[g_{2},g_{3}]]& \text{ if }
i=i_0,\\
-&\text{ otherwise}.
\end{cases}
\]
Hence $\beta_n \in N_{n, \overline{\QQ}}^{i_0}$ as claimed.

To show that $\beta_n$ is non-trivial, it suffices to show that
$[g_{1,1},[g_{2,1},g_{3,1}]]$ is non-trivial.  This is an explicit
calculation using the generating system in Lemma \ref{lem:gensys}(2). 
For instance, one checks that $e_3$ is not sent to itself.
  We conclude that $\beta_n$ is non-trivial, and Claim 1 follows.

\bigskip
\textbf{Claim 2}: $N_{n, \overline{\QQ}}^{i_0}$ is a normal subgroup of
$A_d$.
 
By Lemma~\ref{lem:Ni}, $N_{n,\QQb}^{i_0}$ is a subgroup of $A_d$. It follows from~\eqref{eq:wrconj}
that the conjugates of $\beta_n$ by the elements $\alpha_{j,n}\in N_{n, \overline{\QQ}}$
are also in~$N_{n, \overline{\QQ}}^{i_0}$ for $j\in \{1,2,3\}$.  The group
$G_{1, \QQb}$ is generated by the $g_{j,1}$.  Since $g_{j,1}$ is the
component of $\alpha_{j,n}$ in $N_{n, \overline{\QQ}}^{i_0}$ for  $j\in \{1,2,3\}$,
the group $N_{n, \overline{\QQ}}^{i_0}$ contains the element $\sigma^{-1}
\rho^{i_0} \sigma$ for all $\sigma \in G_{1, \QQb}$. Since $A_d\subseteq
G_{1,\QQb}$ we conclude that $N_{n, \overline{\QQ}}^{i_0}$ is a normal
subgroup of $A_d$. This proves Claim $2$.

As explained in the beginning of the proof, the statement for $d\geq
5$ follows from Claims $1$ and $2$. The remaining cases  can be checked separately. 
\end{proof}

\subsection{Determination of $G_{n, \QQb}$ for normalized Belyi maps}\label{sec:GnQQb}

Let $f$ be a normalized Belyi map of type $(d; e_1, e_2,
e_3)$.  In this section, we completely determine the group structure
of the groups $G_{n, \QQb}$ (Definition~\ref{def:Gn}) as a subgroup of
$\Aut(T_n)$. We refer to Lemma~\ref{lem:gensys} for the description of
$G_{1, \QQb}$.

\begin{theorem}\label{thm:G2alt}
Let $f$ be a normalized Belyi map of type $\underline{C}=(d; e_1, e_2,
e_3)\notin\{(4;3,3,3), (6;4,4,5)\}$.  
\begin{enumerate}
\item Assume that $G_{1,\QQb}\simeq S_d$.  Then
\[
G_{n, \QQb} \simeq E_{n}.
\]
\item Assume that $G_{1,\QQb}\simeq A_d$, i.e., that all $e_j$ are odd. Then 
$G_{n,\QQb}$ is the $n$-fold iterated wreath product of~$A_d$ with itself.
\end{enumerate}
\end{theorem}

The key step in the proof of Theorem \ref{thm:G2alt} is determining
the size of the group $G_{n, \QQb}$.

\begin{lemma}\label{lem:conjH}
Let $f$ be a normalized Belyi map. 
Fix an integer $n\geq 2$ and assume that $E_{n-1} \simeq G_{n-1,
  \QQb}$.  Then
\[
[(S_d)^{d^{n-1}}:N_{n, \overline{\QQ}}]=2^{d^{n-2}}.
\]
\end{lemma}

Note that the assumption that $E_{n-1} \simeq G_{n-1, \QQb}$ for 
$n=1$ states that $G_{1, \QQb}=E_1\simeq S_d$. 

\begin{proof}
Let $\chi$ denote the following homomorphism, induced by the $\sgn$ function on each of the components $N_{n, \overline{\QQ}}^i$: 
\[
\chi: N_{n, \overline{\QQ}} \to \mathbb{F}_2^{d^{n-1}}, \qquad (\rho^i)_{1\leq  i \leq d^{n-1}}\mapsto (\dlog_{-1}(\sgn(\rho^i)))_{1\leq i \leq d^{n-1}},
\]
where the discrete logarithm $\dlog_{-1}$ sends $1\mapsto 0$ and
$-1\mapsto 1$. Proposition \ref{prop:betas} implies that the kernel of $\chi$ equals $ N_{n, \overline{\QQ}}^1\oplus \cdots \oplus N_{n, \overline{\QQ}}^{d^{n-1}} \simeq (A_d)^{d^{n-1}}$. 
Therefore, $\chi$ induces an injection
\[
\overline{\chi}:\overline{N}_{n, \overline{\QQ}}:=N_{n, \overline{\QQ}}/ \left( \oplus_{1 \leq i \leq d^{n-1}}{N_{n, \overline{\QQ}}^i} \right) \hookrightarrow \mathbb{F}_2^{d^{n-1}}.
\]
For the remainder of the proof we identify $\overline{N}_{n, \overline{\QQ}}$ with its image in
$\mathbb{F}_2^{d^{n-1}}$.

Since $e_1+e_2+e_3=2d+1$ is odd and we assume that $G_{1, \QQb}\simeq S_d$, at
least one of the $e_j$ is even. It follows that there is a unique~$j$
such that $e_j$ is odd.  Let $s\in \{1,2, 3\}$ be one of the indices
such that $e_s$ is even. Since at most one of the $e_j$ equals $d$, we
may assume that $e_s\neq d$.

Lemma \ref{alphas}(2) implies that exactly $e_s^{n-1}$ entries of the
element $\alpha_{s,n}\in N_{n, \QQb}$ are $e_s$-cycles, and all other
entries are trivial. It follows that exactly $e_s^{n-1}$ entries of
the corresponding element $\overline{\alpha}_{s,n}\in \overline{N}_{n,
  \overline{\QQ}}$ are non-trivial, i.e.,equal to $1$. Moreover,
dividing the indices once more into $d^{n-2}$ blocks of $d$ indices,
where the $i$th block corresponds to the vertices of $T_{n-1}^i$, we
see that, for each $\overline{\alpha}_{s,n}$, exactly $e_s^{n-2}$ of these blocks contain exactly $e_s$
non-trivial entries. In this proof we denote the $i$th block by $B(i)$. Our choice of numbering implies that\[ 
B(i)=\{1+(i-1)d, 2+(i-1)d, \ldots, d+(i-1)d\},
\]
but we do not need this in what follows.
  A vector $x\in
\FF_2^{d^{n-2}}$ with exactly $e_s^{n-1}$ non-trivial entries
distributed in this way among the blocks is said to \emph{satisfy the
  block condition}.

Equation \eqref{eq:wrconj} implies that conjugation by $G_{n, \QQb}$
yields a well-defined action on $\overline{N}_{n, \overline{\QQ}}$, by
permutation of the coordinates of $x = (x_1, \ldots, x_{d^{n-1}}) \in
\overline{N}_{n, \overline{\QQ}}$ respecting the block structure.  The
assumption that $G_{n-1, \QQb} \simeq E_{n-1}$, together with the
observation that the projection $G_{n,\QQb} \to G_{n-1, \QQb}$ is
surjective, implies that the orbit of $\overline{\alpha}_{s, n}$ under this
action consists of all $x\in \mathbb{F}_2^{d^{n-1}}$ that have exactly
$e_s^{n-1}$ entries equal to $1$ and satisfy the block condition.

For every subset $\mathcal{I} \subseteq \{1,2, \ldots, d^{n-1}\}$ with
$|\mathcal{I}|=e_s^{n-1}$, we denote by $\xi_\mathcal{I}$ the element
of $\mathbb{F}_2^{d^{n-1}}$ such that $(\xi_\mathcal{I})_i = 1$ if $i
\in \mathcal{I}$, and $(\xi_\mathcal{I})_i = 0$ otherwise. With this
notation, the elements in the orbit of $\overline{\alpha}_s$ can be
described as $\xi_\mathcal{I}$ for some $\mathcal{I}$ satisfying the
following properties:
\begin{enumerate} 
   \item For $1\leq i\leq d^{n-2}$, let $\mathcal{I}(i):=\mathcal{I}\cap
     B(i)$. Then $\mathcal{I}(i) \neq \emptyset$ for exactly
     $e_s^{n-2}$ values of $i$, and
   \item if $\mathcal{I}(i) \neq \emptyset$ for some $i$, then
     $|\mathcal{I}(i)|=e_s$. 
\end{enumerate}

Let $\mathcal{H} = \cap_{i=1}^{d^{n-2}} \mathcal{H}_i$ denote the
intersection of all hyperplanes
\[
\mathcal{H}_i:=\{x\in \mathbb{F}_2^{d^{n-1}}\mid \sum_{b\in B(i)} x_{b}=0 \}.
\]
  
\textbf{Claim}: The elements $\xi_\mathcal{I}$ for $\mathcal{I}$
satisfying properties (1) and (2) above generate $\mathcal{H}$.

Fix $k,\ell$ and $i$ such that $1\leq k< \ell \leq d$ and $1\leq
j\leq d^{n-2}$. Let $\xi_{k,\ell;j}$ denote the vector whose entries are
$1$  in the positions $k+(j-1)d$ and $\ell+(j-1)d$ and $0$ otherwise. Note that
$\xi_{k,\ell;j} \in \mathcal{H}$.

To prove the claim it suffices to show that for all $1\leq j\leq
d^{n-2}$ and for all $1\leq k< \ell \leq d$, the vectors $\xi_{k,l;j}$
are in the linear hull of the $\xi_\mathcal{I}$. Using that
$G_{n,\QQb}$ acts transitively on the blocks, it suffices to prove
this for $j=1$. Since $N_{n, \QQb}^1\subseteq G_{n,\QQb}$ acts
transitively on $B(1)$ by Proposition \ref{prop:betas}, we may
moreover assume that $(k, \ell)=(1,2)$. In other words, it suffices to show that $\xi_{1,2;1}\in \mathcal{H}$.

Define 
\[ 
\mathcal{I}(1)=\{1,\ldots,e_{s+1}\}-\{2\} \hspace{2mm} \text{and} \hspace{2mm} \mathcal{I}'(1)=\{1,\ldots,e_{s+1}\}-\{1\}.
\]
For any $2\leq i \leq e_s^{n-2}$, we define the set $\mathcal{I}(i)
= \{b+(i-1)d \mid 1 \leq b \leq e_s\}$.
Then letting $\mathcal{I}= (\cup_{i=2}^{e_s^{n-2}}\mathcal{I}(i)) \cup
\mathcal{I}(1)$ and $\mathcal{I}'=
(\cup_{i=2}^{e_s^{n-2}}\mathcal{I}(i)) \cup \mathcal{I}'(1)$, we see
that $\xi_{1,2;1}=\xi_{\mathcal{I}} +\xi_{\mathcal{I}'}\in \mathcal{H}$.  This proves
the claim.

The claim inductively implies that
\[
[(S_d)^{d^{n-1}}:N_{n, \overline{\QQ}}]=\frac{2^{d^{n-1}}}{|\overline{N}_{n, \overline{\QQ}}|}=2^{d^{n-2}}.
\]
This proves the lemma.
\end{proof}

Now we are ready to prove Theorem~\ref{thm:G2alt}. 

\bigskip\noindent \textit{Proof of Theorem}~\ref{thm:G2alt}.  In
Corollary \ref{cor:GnEn} we have shown that $G_{n, \QQb}$ is a
subgroup of $E_n$ for all $n \geq 1$.  

\textbf{Case 1}: Assume that $G_{1, \QQb}=E_1=\Aut(T_1) \simeq S_d$.  

To prove Statement (1), it suffices to show that the groups $G_{n, \QQb}$ and
$E_n$ have the same cardinality. Since $G_{n, \QQb}$ is a subgroup of
$\Aut(T_n)$ for all $n\geq 1$, the definition of $N_{n, \overline{\QQ}}$ implies  that
\[
\frac{[\Aut(T_n):G_{n, \QQb}]}{[\Aut(T_{n-1}):G_{n-1,
      \QQb}]}=[\left(S_d\right)^{d^{n-1}}:N_{n, \overline{\QQ}}]=2^{d^{n-2}}.
\]
The last equality is the statement of Lemma  \ref{lem:conjH}.
The expression for $[\Aut(T_n):E_n]$ in Lemma~\ref{lem:En}(1) implies that
\[
\frac{[\Aut(T_n):E_n]}{[\Aut(T_{n-1}):E_{n-1}]}=2^{d^{n-2}},
\]
as well.
It follows that $|G_{n, \QQb}|=|E_n|$. This proves the theorem in the first case.

\bigskip
\textbf{Case 2}: Assume that $G_{1, \QQb} \simeq A_d$. 

In this case $G_{n, \QQb}$ is  a subgroup of the $n$-fold
iterated wreath product of $A_d$ with itself.  The statement of the theorem in
this case therefore  follows by induction on $n$ using Proposition
\ref{prop:betas}.

This finishes the proof of the theorem.
\qed

\subsection{Descent}\label{sec:descent}

In this section we determine the groups $G_{n,\QQ}$ (Definition
\ref{def:Gn}) under certain conditions.
Recall from \eqref{eq:GnQQ} that 
\[
G_{n,
  \QQb}\subseteq G_{n,\QQ}\subseteq \Aut(T_n), \qquad \text{ for all }n\geq 1.
\]
Theorem \ref{thm:G2alt} implies that 
\[
[\Aut(T_n):G_{n, \QQb}]=\begin{cases} 2^{d^{n-2}+d^{n-3}+\cdots+d+1}
&\text{ if }G_{1, \QQb}\simeq
S_d,\\ 2^{d^{n-1}+d^{n-2}+\cdots+d+1}&\text{ if }G_{1, \QQb}\simeq A_d
\text{ and }\underline{C}\neq (4;3,3,3).
\end{cases}
\]
If we can show for some normalized Belyi map $f$ that $G_{n,
  \QQb}=G_{n, \QQ}$ for all $n\geq 1$, then we have found an explicit
description of $G_{\infty, \QQ}:=\varprojlim_{n} G_{n, \QQ}$,
and we have shown that this group satisfies $[\Aut(T_\infty):G_{\infty, \QQ}]=\infty$.

\begin{remark}\label{rem:disc}
It is not true in general that $G_{n, \QQb}=G_{n, \QQ}$.
 In fact, this may already fail for $n=1$ when $G_{1,
\QQb}\subsetneq S_d$. Consider a normalized Belyi map $f$ of type
  $(d; e_1, e_2, e_3)$ with all $e_i$ odd. Recall from Lemma
  \ref{lem:gensys}(1) that then $G_{1, \QQb}\simeq A_d$. It follows
  that
\[
A_d\simeq G_{1,\QQb} \subseteq G_{1, \QQ}\subseteq \Aut(T_1)\simeq S_d.
\]
Using the notation from Section \ref{sec:Belyired}, we have that
$G_{1, \QQ}\simeq A_d$ if and only if the discriminant $\Delta(f-t)$ of
$x^{e_1}f_1-tf_2$, considered as polynomial in $x$ over $\mathbb{Q}(t)$, is a square in
$\QQ(t)$, and $G_{1, \QQ}\simeq S_d$ otherwise. For simplicity, we
restrict to the case that $f$ is a polynomial, i.e., that $e_3=d$. In
this case, a formula for the discriminant of $f$ is given in
\cite[Proposition 3.1]{aitken}:
\[
\Delta(f-t)=(-1)^{(d-1)(d-2)/2}d^d \ell(f)^{d-1}t^{e_1-1}(t-1)^{e_2-1},
\]
where $\ell(f)$ denotes the leading coefficient of
$f$. Proposition~\ref{prop:Belyifam}(1) gives an explicit expression
for $\ell(f)$.  The formula for the discriminant, together with the
expression for $\ell(f)$, implies that if $f$ is a polynomial and $d$
is odd, then $\Delta(f-t)$ is never a  square in $\QQ(t)$. We note that a
formula for $\Delta(f-t)$ in the non-polynomial case may be deduced
from \cite[Proposition~1]{CulFar}. 
\end{remark}

The structure of $\Aut(T_n)$ as an iterated wreath product implies that
 the case $n=2$ plays a key role in determining $G_{\infty,\QQ}$. 
 Recall that $f^n$ denotes the $n$th iterate of $f$. 
   
 \begin{proposition}\label{prop:leveltwo}
Let $f$ be a normalized Belyi map of type $\underline{C} = (d; e_1,
e_2, e_3)$. 
If $G_{2, \QQ}=G_{2, \QQb}$, then
$G_{n, \QQ}=G_{n, \QQb}$ for all $n\geq 2$.
\end{proposition}

 \begin{proof}
 We first assume that $G_{1, \QQb}=S_d$.
We prove the statement by induction on $n \geq 2$. Assume that the statement holds for $n-1$, i.e., we have $G_{n-1, \QQ}=G_{n-1,\QQb}$.
Theorem \ref{thm:G2alt}(1) implies that
\[
E_n=G_{n, \QQb}\subseteq G_{n, \QQ}\subseteq \Aut(T_n).
\]
The induction hypothesis implies that
\[
G_{n,\QQb}\subseteq G_{n, \QQ}\subseteq G_{n-1, \QQ}\wr G_{1, \QQ}=G_{n-1,
  \QQb}\wr G_{1, \QQb}=E_{n-1}\wr E_1.
\]
The second inclusion follows from the structure of $G_{n, \QQ}$ as a 
wreath product induced by the decomposition $f^n=f\circ f^{n-1}$.
Lemma \ref{lem:En}(1) implies that $[E_{n-1}\wr E_1:E_n]=2$. We
conclude that $G_{n, \QQ}$ equals either $G_{n, \QQb}=E_n$ or
$E_{n-1}\wr E_1$. The definition of the wreath-product sign
  (Definition \ref{def:sgn2}) implies that we can  distinguish
  the two possible groups by considering their images under $\pi_2$. Since
  $\pi_2(G_{n, \QQ})=G_{2, \QQ}$ and $\pi_2(G_{n, \QQb})=G_{2, \QQb}$, the result follows in this case.
  
  Now assume that $G_{1, \QQb}=A_d$.  Theorem \ref{thm:G2alt}(2) implies that $G_{n, \QQb}$ is the $n$-fold iterated wreath product of $A_d$ with itself. We denote this group by $E_n^+$. As in the previous case, we conclude from the induction hypothesis that
  \[
  G_{n,\QQb}\subseteq G_{n, \QQ}\subseteq G_{n-1, \QQ}\wr G_{1, \QQ}=G_{n-1,
    \QQb}\wr G_{1, \QQb}=E_{n-1}^+\wr E_1^+.
  \]
   Note that the wreath product sign restricted to $E_n^+$ is trivial. The statement follows as in the first case.
  \end{proof}

We write $f(x)=g(x)/h(x)$, where $g,h \in \ZZ[x]$ are relatively prime as polynomials in $\ZZ[x]$.  
We define the discriminant of the rational map $f(x)-t$ as the discriminant of the polynomial $g(x)-th(x)$, viewed as a polynomial in $x$ over $\QQ(t)$. In other words, $\Delta(f-t)$ is in $\QQ[t]$ and
 \[ \Delta(f-t):=\Delta(g(x)-th(x)). \] 
 
 Let $L$ be the splitting field of $g(x)-th(x)$ over $\Q(t)$ and let $t_i$ for $i=1,\ldots,d$ be the roots of $g(x)-th(x)$ in $L$.

\begin{lemma}\label{lem:sgncharacter}
Let $f(x)$ be a rational map in $\QQ(x)$. The Galois group $G_{2,\QQ}$ (attached to $f$) is a subgroup of $\ker(\sgn_2)$ if and only if 
the product $\Delta(f(x)-t)\prod_i{\Delta(f(x)-t_i)}$ is a square in $\QQ(t)$. 
\end{lemma}
\begin{proof}
Let $\sigma=((\sigma_1,\ldots,\sigma_k), \tau)$ be in $G_{2,\QQ} \leq \Aut(T_2)$ and let $D=\Delta(f(x)-t)\prod_i{\Delta(f(x)-t_i)}$.
Then we will show that $\sigma(\sqrt{D})=\sgn_2(\sigma)\sqrt{D}$. 

Let $t_{ij} $, for $1\leq j\leq d$, denote the roots of $f(x)-t_i$ for
all $i$. We can identify $t_{ij}$ with the vertex $(i,j)$ of the tree
$T_2$. By \eqref{eq:wraction} the action of $\sigma$ on $t_{ij}$ is
given by
   \[
   \sigma(t_{ij})=t_{\tau(i)\sigma_i(j)}. 
    \]
 We note here that, for $i=1,\ldots,d$, $\sqrt{\Delta(f-t_i)}$ lives in some quadratic extension of L. Hence 
 \[
   \sigma(\sqrt{\Delta(f-t_i)})=\sigma(\prod_{j<k}(t_{ij}-t_{ik}))=\prod_{j<k}(t_{\tau(i)\sigma_i(j)}-t_{\tau(i)\sigma_i(k)})=\sgn(\sigma_i)\sqrt{\Delta(f-t_{\tau(i)})}.
 \]   
 Similarly, $\sigma(\sqrt{\Delta(f-t)})=\sgn(\tau)\sqrt{\Delta(f-t)}$ and hence
  \[
     \sigma(\sqrt{D})=\sgn_2(\sigma)\sqrt{D}.
  \]
    This concludes the proof.
\end{proof}

\begin{corollary}\label{cor:condsgn}
Let $f$ be a normalized Belyi map of type $\underline{C} = (d; e_1,
e_2, e_3)$. Assume that $G_{1,\QQb}=S_d$. Then $G_{n,\QQ}=G_{n,\QQb}$ for all $n \geq 1$ if and only if the product $\Delta(f(x)-t)\prod_i{\Delta(f(x)-t_i)}$ is a square in $\QQ(t)$.
 \end{corollary}

Next, we will compute the discriminant product in Lemma~\ref{lem:sgncharacter} for a dynamical Belyi map of type $\underline{C}=(d;e_1,e_2,e_3)$. Let us write
$\Delta(f(x)-t)=a(f)t^{e_1-1}(1-t)^{e_2-1}$ with $a(f) \in \Q$. 

\begin{proposition}\label{prop:discriminantsgn}
Let $f(x) = g(x)/h(x)$ be a normalized Belyi map of type $\underline{C} = (d; e_1,
e_2, e_3)$. Then 
   \[
       \Delta(f(x)-t) \prod_i{\Delta(f(x)-t_i)} = u(1-t)^{2(e_2-1)}t^{2(e_1-1)},
   \]
 where $u=(-1)^{(d+1)(e_1-1)}a(f)^{d+1}h(0)^{e_1-1}g(1)^{e_2-1}/\ell(g)^{e_1+e_2-2}$.
\end{proposition}
\begin{proof}
For $f(x)=g(x)/h(x)$ as above we have $\Delta(g(x)-t_ih(x))=a(f) t_i^{e_1-1}(1-t_i)^{e_2-1}$ and hence
 \[ 
  \prod_i \Delta(g(x)-t_ih(x))=a(f)^d\prod_{i=1}^d(t_i)^{e_1-1}\prod_{i=1}^d(1-t_i)^{e_2-1}. 
 \]
 Using the fact that $g(x)-th(x)=\ell(g)\prod_{i=1}^d(x-t_i)$, substituting $x=0$ and $x=1$, we compute that 
\[
 \prod_{i=1}^d {t_i}=\frac{(-1)^{d+1}h(0)}{\ell(g)}t \quad \text{and} \quad  \prod_{i=1}^d {(1-t_i)}=\frac{g(1)}{\ell(g)}(1-t). 
 \]
Therefore we find that
\begin{align*}
  \Delta(f(x)-t) \prod_i{\Delta(f(x)-t_i)}=& a(f)^{d+1}t^{e_1-1}(1-t)^{e_2-1} \prod_{i=1}^d(t_i)^{e_1-1}\prod_{i=1}^d(1-t_i)^{e_2-1} \\
  =& a(f)^{d+1}t^{e_1-1}(1-t)^{e_2-1}(\frac{(-1)^{d+1}h(0)}{\ell(g)}t)^{e_1-1} (\frac{g(1)}{\ell(g)}(1-t))^{e_2-1},
\end{align*}
as claimed.
\end{proof}

\begin{corollary} \label{cor:descent} \hfill
\begin{enumerate}
   \item Let $d$ be odd.
\begin{enumerate}
\item Let $f$ be a normalized Belyi map of type $(d; e_1, e_2,
  e_3=d)$ such that  $e_1$ or $e_2$ is even. Then $f$ is a polynomial such that $G_{1,\QQb}=S_d$ and  $G_{n,\QQ}=G_{n,\QQb}$ for all $n \geq 1$.
  \item Let $f$ be a normalized Belyi map of type $(d; d-k, 2k+1, d-k)$ with $k$ odd.  Then $G_{1,\QQb}=S_d$ and $G_{n,\QQ}=G_{n,\QQb}$ for all $n \geq 1$.
\end{enumerate}
\item Let $f$ be a normalized Belyi map such that $G_{1,\QQb}=G_{1,\QQ}=A_d$. Then $G_{n,\QQ}=G_{n,\QQb}$ for all $n \geq 1$.
\end{enumerate}
\end{corollary}
\begin{proof}
\begin{enumerate}
\item  
Proposition~\ref{prop:discriminantsgn}, together with  the explicit expression for $f$ in Proposition~\ref{prop:Belyifam} and the assumption that $d$ is odd, implies in both cases that 
$\Delta(f(x)-t)\prod_i{\Delta(f(x)-t_i)}$ is a square. 
(In the situation of Statement (a) one has that $h(0)=g(a)=1$. In the situation of Statement (b) one has that $h(0)=\ell(g)=a_0$.  

The statement follows therefore from Corollary~\ref{cor:condsgn}.
\item We have that $G_{n,\QQ}$ is a subgroup of the $n$-fold iterated wreath product of $A_d$ with itself, since it can be identified with a subgroup of $G_{n-1,\QQ} \wr G_1$. Hence the result follows from Theorem~\ref{thm:G2alt} and the fact that $G_{n,\QQb} \subset G_{n,\QQ}$.
\end{enumerate}
\end{proof}

\begin{remark}\label{rem:descent} \hfill
\begin{enumerate} 
\item One may treat the case that $d$ is even using the formula in Proposition~\ref{prop:discriminantsgn}, though the statements are not quite as nice in this case. For instance, let $f(x)=-(d-1)x^d+x^{d-1}$ for $d \geq 4$. Note that $f$ has combinatorial type $(d; d-1, 2 ,d)$. Assume that $d$ is even and $G_{1,\QQb}=S_d$. Then we find that
$G_{2,\QQ}=G_{2,\QQb}$ if and only if $d$ is divisible by $4$.

\item
Assume that $G_{1,\overline{\mathbb{Q}}}=A_d$ and that $G_{1,\QQ}=S_d$. By Proposition~\ref{prop:discriminantsgn} and Lemma~\ref{lem:sgncharacter}, the group $G_{2,\QQ}$ is contained in $\ker(\sgn_2)$ if and only if $a(f)^{d+1}$ is a square in $\QQ$. Since $G_{1,\QQ}=S_d$, we know that $a(f)$ is not a square. It follows that $G_{2,\QQ}$ is contained in $\ker(\sgn_2)$ if and only if $d$ is odd.
\end{enumerate}

\end{remark}

\section{Specialization}\label{sec:spec}

  In this section we prove some explicit results on the specialization
  of normalized Belyi maps $f$.  For any $n \geq 1$, Hilbert's
  Irreducibility Theorem implies that there exists a non-empty
  Zariski-open set $\mathcal{H}_n=\mathcal{H}_n(f) \subseteq
  \PP^1(\QQ)$, called a \emph{Hilbert set}, such that specializing the
  parameter $t$ to $a\in \mathcal{H}_n$ does not change the Galois
  group. In this section we determine explicit elements $a\in
  \mathcal{H}_n(f)$ for all $n$. This means that we get an explicit
  tower of number fields $(K_{n,a})_{n \geq 1}$ with prescribed Galois groups (Definition
  \ref{def:Kn}) by specializing to these values of $a$.

  These elements are determined by 
  Conditions \ref{conditions}. Conditions \ref{conditions} can be
  thought of as the analogue of Conditions $(\dagger)$ of
  \cite{BFHJY}, and Proposition \ref{prop:localEisenstein} as the
  analogue of \cite[Proposition 3.4]{BFHJY}.

\subsection{Irreducibility and ramification conditions}

Throughout this section, we fix a normalized Belyi map $f$ of type
$\underline{C}:=(d; e_1, e_2, e_3)$.
Recall that
  we write $f$ both for the rational function $f(x)\in \QQ(x)$ and the
  corresponding map $f: \mathbb{P}^1 \to \mathbb{P}^1 : x \mapsto t =
  f(x)$. As in the beginning of Section
  \ref{sec:Belyired} we write $f(x)= x^{e_1}f_1(x)/f_2(x)$ and
  assume that the polynomials $f_i$ satisfy (1)--(3) introduced there.

Recall that for any $n \geq 1$ we write $F_n/\QQ(t)$ for the extension
of function fields corresponding to the map $f^n:\PP^1_{\QQ}\to
\PP^1_{\QQ}$ and that $G_{n, \QQ}$ denotes the Galois group of the
Galois closure of this extension (Definition \ref{def:Gn}).

\begin{definition}\label{def:Kn}
Let $a\in \PP^1(\QQ)\setminus\{0,1,\infty\}$ such that the numerator
$f(n,a)$ of $f^n-a$ is irreducible for all $n\geq 1$ and
define $K_{n,a}$ as the extension of $K_{0,a}:=\QQ$ obtained by
adjoining a root of $f(n,a)$.  We denote
by $G_{n,a}$ the Galois group of the normal closure of
$K_{n,a}/K_{0,a}$.
\end{definition}

\begin{proposition}\label{prop:Eisenstein} 
Let $f$ be a normalized Belyi map of type $\underline{C}:=(d; e_1,
e_2, e_3)$.  Assume $f$ has good monomial reduction at $p$ for some
prime $p$.  Choose $a\in \PP^1(\QQ)\setminus\{0,1,\infty\}$ with
$\nu_p(a)=1$.  Then 
\[
[K_{n,a} : \QQ] = d^n\qquad \text{ for all } n \geq 1.
\]
  In particular, $G_{n,a}$ is a transitive subgroup of $S_{d^n}$ for
  all $n \geq 1$.
\end{proposition}

Recall that an explicit criterion for good monomial
reduction was given in Proposition~\ref{prop:Belyiredcrit}(2).

\begin{proof}
Write $f(x)=x^{e_1}f_1(x)/f_2(x)$ with $f_1(x)=\sum_{i=0}^{d-e_1}
a_ix^i$ and $f_2(x)=\sum_{i=0}^{d-e_3} b_ix^i$ both in $\ZZ[x]$. Since
$f$ satisfies (1)--(3) from Section \ref{sec:Belyired} and we assume
it has good monomial reduction at $p$, we have that
\begin{itemize}
\item $\nu_p(a_i)=\nu_p(b_j)=0\text{ for $i=d-e_1$ and $j=0$}$,
\item $\nu_p(a_i)>0 \text{ for } i\neq d-e_1$,
\item $\nu_p(b_j)>0 \text{ for } j\neq 0$.
\end{itemize}
We conclude that the numerator of
\[
f_a(x):=f(x)-a=\frac{x^{e_1}f_1(x)-a f_2(x)}{f_2(x)}
\]
is an Eisenstein polynomial for $p$, hence irreducible. Here we have also used
that $f(0)=0$ and $\nu_p(a)~=~1$. Moreover, the numerator and the
denominator of $f_a(x)$ are relatively prime.  The statement follows
for $n=1$.

Similarly, for $n>1$ arbitrary we find
\[
f^n(x)=\frac{a_{d-e_1}^{1+d+\cdots+d^n}x^{d^n}+p\cdot(\text{terms of degree } <
  d^n)}{p\cdot (\text{terms of degree } \geq 1) +b_0^{1+d+\cdots +d^n}}.
\]
The same argument as for $n=1$ therefore also applies to the case of arbitrary $n$.
\end{proof}

For $a\in \PP^1(\QQ)\setminus\{0,1,\infty\}$ such that the numerator
of $f^n-a$ is irreducible, i.e., such that
$[K_{n,a}:\QQ]=[F_n:\QQ(t)]=d^n$, there exists an isomorphism between
the field extensions $\left(K_{n,a}\otimes_\QQ \QQ(t)\right)/\QQ(t)$ and
$F_{n}/\QQ(t)$. This isomorphism induces an inclusion of Galois groups
\begin{equation}\label{eq:Galin}
G_{n,a}\hookrightarrow G_{n,\QQ}.
\end{equation}
In the rest of this section we fix these inclusions for all $n$. For
more details we refer to \cite[Sections 1.1.1 and 1.1.2]{volklein}.

Proposition \ref{prop:localEisenstein} below provides partial
information on the ramification of $K_{n,a}/\QQ$ for suitable choices
of $a$. Here we use the reduction of $f$ at a prime $q$ for which $f$
has good separable reduction.  In Proposition
\ref{prop:Belyiredcrit}(1) we showed that this holds if $q>d$.  The
idea of the proof of Proposition~\ref{prop:localEisenstein} is similar
to that of Proposition \ref{prop:Eisenstein}.  What we show is that if
$a\equiv 0\pmod{q}$, then the ramification of $q$ in $K_{n, a}/\QQ$ is
the same as the ramification above $t=0$ in the iteration
$f^n:\PP^1\to \PP^1$, which was described in Remark
\ref{rem:ramfn}. We give similar statements for the ramification
above the other branch points $t=1, \infty$.  From this, we deduce the
existence of concrete elements in $G_{n,a}$; see Lemma \ref{lem:elts}
below for the precise statement.

\begin{proposition}\label{prop:localEisenstein}
Let $f$ be a normalized Belyi map of type $(d; e_1, e_2, e_3)$ and let
$q$ be a prime of good separable reduction for $f$. Let $a\in
\PP^1(\QQ)\setminus\{0,1,\infty\}$ such that $[K_{n,a}:\QQ]=d^n$ for
all $n\geq 1$. Assume additionally that $\nu_q(a)>0$.

Then there is a unique sequence $\{q_n\}_{n \geq 0}$ such that $q_0 = q$ and $q_n$ is a prime ideal of $K_{n,a}$ lying above $q_{n-1}$, satisfying the following properties: the ramification index of $q_n \in K_{n,a}/K_{n-1,a}$ is $e_1$, and all other primes of $K_{n,a}$ lying above $q$ are unramified in $K_{n,a}/K_{n-1,a}$.

The analogous statement with $e_1$ replaced by $e_2$ (resp.~$e_3$)
holds if we require $\nu_{q}(1-a)>0$ (resp.~$\nu_{q}(a)<0$) instead. 
\end{proposition}

\begin{proof}
We only prove the statement in the case that $\nu_q(a)>0$; the other
two cases are similar. 

The assumption that $f$ has good separable reduction at $q$ implies
that this reduction satisfies
\[
\overline{f}=\frac{x^{e_1} \overline{f}_1}{\overline{f}_2},
\]
where $\overline{f}_1$ and $\overline{f}_2$ are separable polynomials
of degree $d-e_1$ and $d-e_3$, respectively, which are relatively
prime.  The ramification points of $f$ are exactly $0,1,
\infty$. Since the reduction $\overline{f}$ of $f$ at $q$ is assumed
to be separable, $\overline{f}$ is exactly ramified at $0,1,\infty$.
Proposition \ref{prop:Belyired}(2) implies that $\overline{f}$ is
also of type $(d; e_1, e_2, e_3)$. It follows that
$\overline{f}_1(0)\neq 0$ and $\overline{f}_2(0) \neq 0$.  We conclude
that there is a unique prime $q_1$ of $K_{1,a}$ above $q_0:=q$ that is
ramified. Moreover, the ramification index of this prime is $e_1$.

The description of the ramification of the map $\overline{f}$ implies
that the Newton polygon of $x^{e_1}{f}_1-a{f}_2$ has two $q$-adic
slopes: a slope $0$ with multiplicity $d-e_1$ and a slope
$-\nu_q(a)/e_1$ with multiplicity $e_1$. Here we use that $e_1
>d-e_3=\deg(f_2)$ and that the leading and the constant coefficients of
$f_1$ and $f_2$ are both $q$-adic units. 

The choice of $q_1$ implies that there exists a root $a_1\in K_{1,a}$
of $f(x)-a$ that satisfies $\nu_{q_1}(a_1)=\nu_q(a)>0$.  We conclude
that the prime $q_1$ and the rational function $f(x)-a_1\in
K_{1,a}(x)$ satisfy the hypothesis of the proposition, as well. By
induction, we conclude that there exists a unique ramified prime $q_n$
in $K_{n,a}/K_{n-1,a}$. Moreover, since we assume that
$[K_{n,a}:K_{n-1, a}]=d$, the prime $q_n$ lies above the prime
$q_{n-1}$, and its ramification index is $e_1$.

Let $q_1'$ of $K_{1,a}$ be an unramified prime above $q_{0}$. In this
case it follows that $\nu_{q_1'}(a_1)=0$. We conclude that all primes
above $q_1'$ in $K_{2,a}$ are unramified in $K_{2,a}/K_{1,a}$. By
induction, we conclude that if a prime $q_n'$ above $q$ is unramified
in $K_{n,a}/K_{n-1, a}$ then all primes above it in the tower of
number fields are unramified. This concludes the proof.
\end{proof}

\begin{remark} In the proof of Proposition \ref{prop:localEisenstein} we use
 that fact that if $f$ has good separable reduction at a prime~$q$
 then the reduction $\overline{f}$ has the same type as the map $f$.
 This property is very specific to the case of normalized single-cycle
 genus-$0$ Belyi maps and does not hold without these assumptions on the map. We give
 an example to illustrate this. 

 Let $p>3$ be a prime. The map $f(x)=x^2(x-p)$  has type $(3;
 2,2,3)$, but it is not normalized in the sense of Definition
 \ref{def:type}(2). The branch points are also not normalized to
 $0,1,\infty$. The two ramification points with ramification index $2$
 are $0$ and $p$, which specialize to the same point modulo $p$, and this
 point has ramification index $3>2$. In our very special setting this
 cannot happen.
\end{remark}

Propositions \ref{prop:Eisenstein} and \ref{prop:localEisenstein}
suggest the following conditions.

\begin{conditions}\label{conditions}
Let $f$ be a normalized Belyi map of type $(d; e_1, e_2, e_3)$. Choose $a \in \mathbb{P}^1(\mathbb{Q})\setminus \{0,1,\infty\}$ and distinct
primes $p, q_1, q_2, q_3$ such that the following hold: $f$ has good
monomial reduction at $p$ and good separable reduction at $q_1, q_2,
q_3$, and we have
\[
\nu_p(a)=1,\qquad  \nu_{q_1}(a)>0, \qquad \nu_{q_2}(1-a)>0, \qquad \nu_{q_3}(a)<0.
\]
\end{conditions}

\begin{remark}\label{rem:tower}
The results in this paper can be used to construct towers of number
fields that are branched over an explicit finite set of primes.  Fix
a normalized Belyi map $f$ of type $(d; e_1, e_2, e_3)$ and a value $a
\in \mathbb{P}^1(\QQ)\setminus \{0,1,\infty\}$ such that $[K_{n,a}:
  \QQ] = d^n$ for all $n \geq 1$. Construct the tower $\QQ=K_{0,a}
\subseteq K_{1,a} \subseteq K_{2,a} \subseteq \ldots$ of number
fields.  We denote the set of rational primes in $\QQ$ by~$\PP$. Recall from the introduction that there is a finite set
$\mathcal{P} \subseteq \PP$ such that $K_{n,a}/K_{0,a}$ is unramified
outside primes lying above $\mathcal{P}$. (This is \cite[Theorem
  1]{CulFar}, using that normalized Belyi maps are post-critically
finite.)

We sketch what we can say about the finite set $\mathcal{P}$ in our
situation.  (This is a more precise version of \cite[Section
  5]{CulFar}, using the results on the reduction of normalized Belyi
maps from \cite{WINE2}.) A subtle point is that there is a difference
between the reduction of a rational function $f\in \ZZ[x]$ (defined by
reducing the coefficients modulo $p$ as in Definition
\ref{def:Belyired}) and reduction of the cover $f:\PP^1_\QQ\to
\PP^1_\QQ$. A model over $\Spec(\ZZ_p)$ of the cover $f\otimes_\QQ
\QQ_p$ is required to be finite and flat.  Reducing the rational
function $f\in \ZZ[x]$ modulo $p$ by reducing its coefficients may
yield a rational function of strictly smaller degree. This happens if
and only if $f$ has bad reduction to characteristic $p>0$ in the sense
of Definition \ref{def:goodred}.

The results of \cite[Section 4]{WINE2} can be interpreted as saying
that these notions are closely related for normalized Belyi maps in
the single-cycle case. Namely, the rational function $f$ has
good separable reduction at $p$ in the sense of Definition \ref{def:goodred} if
and only if the Galois closure of the map $f:\PP^1_{\QQ_p}\to
\PP^1_{\QQ_p}$ has potentially good reduction, meaning that there
exist an extension $L/\QQ_p$ and a model of $f\otimes_{\QQ_p} L$ over
$\Spec(\mathcal{O}_L)$ whose special fiber is a separable Galois
cover of $\PP^1$ branched over three points.  This is very special to the
case of normalized single-cycle Belyi maps. 

Recall that we have chosen $f\in \ZZ(x)$ so that the conditions
(1)--(3) in Section \ref{sec:Belyired} are satisfied. As in 
\cite[Section 5]{CulFar} we let
$\mathcal{P}_\text{bad}$ be the set of rational primes for which the
fiber of $f$ at $p$ has degree strictly less than $p$ or is
inseparable. (The third case of \cite[Section 5]{CulFar}, in which the ramification points coalesce
modulo $p$, does not occur in our case, by Proposition
\ref{prop:Belyired}(2).)
  Proposition \ref{prop:Belyiredcrit}(1) implies that
\[
\mathcal{P}_\text{bad}\subseteq \{ p\in \PP \mid p\leq d\}.
\]
We can determine this set more precisely for a given type:
$\deg(\overline{f})<\deg(f)$ at $p$ if and only if $p$ divides the
leading coefficient of $f$. If $f$ has good inseparable reduction at
$p$, then $p\vert\deg(f)$. In fact one can show that $f$ has either
bad or good inseparable reduction at the primes $p$ dividing
$\deg(f)$. (This may be deduced from \cite[Proposition~5]{WINE2}.)

It follows from \cite[Theorem 2]{CulFar} that we may take
$\mathcal{P}=\mathcal{P}_\text{bad}\cup \mathcal{P}_a$, where 
\[
\mathcal{P}_a=\{ p\in \PP\mid  v_p(a)\neq 0 \text{ or } v_p(1-a)>0\},
\]
i.e., the set of primes $p$ such that $a$ is congruent modulo $p$ to one of the
branch points $\{0, 1, \infty\}$ of $f$.

As a concrete example, consider the polynomial Belyi map of type $(d;
d-1, 2, d)$. Then $f(x)=-(d-1)x^d+dx^{d-1}$ has
good monomial reduction at all primes dividing $d$ and bad reduction
exactly at the primes dividing $d-1$,
i.e., $\mathcal{P}_\text{bad}=\{p\in \PP \mid p\vert d(d-1)\}$.
Choosing $a=p$ yields a tower of number fields only branched over
$\mathcal{P}_\text{bad}$.  

The Chinese Remainder Theorem implies that we may also choose $a$ such
that the Conditions \ref{conditions} are satisfied.  For example for
$d = 9$ we may choose $a = 60/11$, and we find
$\mathcal{P}_{\text{bad}} = \{ 2, 3 \}$ and $\mathcal{P}_a = \{ 5, 7,
11 \}$.  The infinite tower of number fields $(K_{n, a})_{n\geq 1}$
corresponding to $f$ only ramifies above $\{2,3,5,7,11\}$.  Combining
Corollary \ref{cor:index} from the next section with Corollary~\ref{cor:condsgn}
and \ref{cor:descent} yields that
$G_{n,a,\QQ}=E_n$ for all $n\geq 1$.

Using the explicit expressions in Proposition \ref{prop:Belyifam} it is
easy to find many more results along these lines.
\end{remark}

Lemma \ref{lem:elts} below translates Conditions \ref{conditions}
into a statement on the existence of certain elements 
$h_{j,n,a}\in G_{n,a}$. We start by setting up some notation. Since $G_{n,a}\subseteq
G_{n, \QQ}\subseteq \Aut(T_n)$, we may define subgroups of $G_{n,a}$
analogous to the subgroups $N_{n, \QQb}$ and $N_{n, \QQb}^i$ defined
in Definitions \ref{def:NNi} and \ref{def:Ni}.

\begin{definition}\label{N2a}
Define  $N_{n,a} := \ker(G_{n,a}\to G_{n-1, a})$ and  $N_{n,a}^i :=
N_n^i \cap \mathcal{S}(T_n^i)$.
\end{definition}

Analogous to \eqref{eq:rhodef}, we may write elements of $N_{n,a}$ as
tuples $(\rho^1, \ldots, \rho^{d^{n-1}})$, where $\rho^i\in
N_{n,a}^i\subseteq \mathcal{S}(T_n^i)$.

\begin{lemma}\label{lem:elts}
Let $f$ be a normalized Belyi map satisfying Conditions
\ref{conditions} for a choice of $a, p, q_1, q_2, q_3$. For $n\geq 2$
and $j\in \{1,2,3\}$ there exist elements $h_{j,n}\in G_{n,a}$ such
that the following conditions hold.
\begin{enumerate}
\item The elements $h_{j,n}\in G_{n,a}$ are conjugate to $g_{j,n}$ in $G_{n,a}$.
\item The element $h_{j,1}\in G_{1,a}$ is a single cycle of length $e_j$.
\item For $n\geq m$, we have that $\pi_m(h_{j,n})=h_{j,m}$.
\item We have
\[ 
\alpha_{j,n,a}:=h_{j,n}^{e_j^{n-1}}\in N_{n,a},
\]
where the components $\alpha_{j,n,a}^i$ of $\alpha_{j,n,a}$ in
$\mathcal{S}(T^i_n)$ for $1\leq i\leq d^{n-1}$ are either single cycles
of length $e_i$ or trivial. The permutation $\alpha_{j,n,a}^i$ is non-trivial for
exactly $e_j^{n-1}$ values of $i$.
\end{enumerate}
\end{lemma}

\begin{proof}
For $n \geq 1$ and $j\in \{1,2,3\}$ let $h_{j,n}\in G_{n,a}$ be a
generator of the inertia group of $q_j$ in $G_{n,a}$. Then Proposition
\ref{prop:localEisenstein} implies that $h_{j,n}$ is conjugate in
$G_{n,a}$ to $g_{j,n}$. Hence Statement (1) holds.  For $n=1$, the
$h_{j,1}$ are single $e_j$-cycles for $j = 1,2,3$, proving  Statement (2). 
Moreover, it is clear that we may choose the $h_{j,n}$ for varying $n$
consistingly, guaranteeing  that Statement (3) holds.

Arguing as in the proof of Lemma \ref{alphas}, we conclude that the
elements $\alpha_{j,n,a}$ are contained in $N_{n,a}$ for $j\in
\{1,2,3\}$. Statement (1) implies that $\alpha_{j,n,a}$ has the same
cycle type as $\alpha_{j,n}$ (defined in Lemma \ref{alphas}) when
considered as an element of $S_{d^n}$. Statement (4) follows therefore
immediately from Lemma~\ref{alphas}.
\end{proof}

\subsection{Comparing $G_{n,a}$ and $G_{n, \QQ}$}

In this section we compare the Galois groups $G_{n,a}\subseteq G_{n, \QQ}$ and
give sufficient conditions on $a$ for these groups to be equal for all
$n\geq 1.$ The key step is to show that the geometric Galois group
$G_{n,\QQb}$ is a subgroup of $G_{n,a}$ for all $n$ if Conditions
\eqref{conditions} are satisfied. We show this using the explicit
elements of $G_{n,a}$ we produced in Lemma \ref{lem:elts} and by arguing
as in Sections \ref{sec:gensys} and \ref{sec:GnQQb}.

\begin{proposition}\label{prop:condEisenstein}
Let $f = x^{e_1}f_1/f_2$ be a normalized Belyi map of type $(d;e_1,
e_2, e_3)\neq (6;4,4,5)$. Assume $f$ satisfies Conditions \ref{conditions} for a
choice of $a, p, q_1, q_2, q_3$.
\begin{enumerate}
\item We have that $G_{1, \QQb}\subseteq G_{1,a} \subseteq S_d$. In particular,
  $G_{1, a}\simeq S_d$ in the case that $G_{1, \QQb}\simeq S_d$.
\item For $n \geq 2$ and $1\leq i\leq d^{n-1}$, 
  the image of the projection map 
\[
N_{n,a} \to \mathcal{S}(T^i_{n}),\qquad  (\rho^1,
  \ldots, \rho^{d^{n-1}}) \mapsto \rho^i
\]
 contains
  $G_{1, \QQb} \subseteq S_d$ as a subgroup. 
\end{enumerate}
\end{proposition}

\begin{proof}
The existence of the prime $p$ of good monomial reduction implies that
the Galois group $G_{1,a}\subseteq S_d$ of $K_{1,a}/\QQ$ is a transitive
group on $d$ letters (Proposition \ref{prop:Eisenstein}). The
conditions on $a$ with respect to the primes $q_i$ imply that
$G_{1,a}$ contains elements $h_{1, 1}, h_{2,1}, h_{3,1}$, which are pure cycles of
length $e_1, e_2, e_3$ respectively, where $e_1 + e_2 + e_3 = 2d+1$
(Lemma  \ref{lem:elts}(2)).

We argue as in the proof of~\cite[Theorem 5.3]{liuosserman} to show that
$G_{1,a}$ contains a subgroup isomorphic to~$G_{1,\QQb}$. We start by proving
that $G_{1,a}$ is primitive.  To reach a contradiction, suppose that
$G_{1,a}$ is\textit{ not} primitive.  Since $G_{1,a}$ is transitive
on $d$ letters, there exists a number $m \vert d$ (with $1 < m < d$)
and a division of $\{1, 2, \ldots, d\}$ into $d/m$ blocks of length
$m$ such that every element of $G_{1,a}$ either has order
strictly less than $m$ and acts trivially on the blocks, or has order
$mk$ for some $k \geq 1$.  Now we distinguish the following cases.
\begin{itemize}
\item If all $e_i$ are strictly less than $m$, then $2d+1=e_1+e_2+e_3
  < 3m \leq 3d/2$, since $m\leq d/2$, by assumption.  We obtain a contradiction.
\item If all  $e_i$ are a multiple of $m$, then
  $2d+1 = e_1 + e_2 + e_3$ is divisible by $m$. Since $m\vert d$ we obtain  a
  contradiction. 
\item Assume that exactly one of the $e_i$ is strictly less than $m$;
  this is necessarily $e_1$, since $1 < e_1 \leq e_2 \leq e_3$.  Write
  $e_2 = mk_2$ and $e_3 = mk_3$ for some $k_2,k_3 \geq 1$. We obtain $2d + 1
  = e_1 + m(k_2+ k_3)$. Since $m\vert d$, we have $e_1 \equiv 1 \pmod m$.
  This implies that $e_1>m$, and we obtain a contradiction.
\item If exactly two of the $e_i$ (that is, $e_1$ and $e_2$) are
  strictly less than $m$, then a similar argument shows that $e_1 +
  e_2 \equiv 1 \pmod m$ so $e_1 + e_2 = 1 + m\leq 1+d/2$. But then $e_3 = (2d+1)
  - (e_1+e_2) \geq 3d/2>d$, which again yields a  contradiction.
\end{itemize}

Hence $G_{1,a}$ is primitive, as claimed.  For $d > 10$, the group
$G_{1,a}$ contains a cycle of length $e \leq (d-e)!$ (namely,
$h_{1,1}$). Hence, by Williamson's Theorem \cite{Williamson}, we have
that $G_{1,a}$ is isomorphic to either $A_d$ or $S_d$.  If at least one
of the $e_i$ is even, then $G_{1,a}$ is isomorphic to $S_d$. In both
cases we therefore have that $G_{1,a}$ contains $G_{1, \QQb}$ (Lemma
\ref{lem:gensys}(1)).  For $d \leq 10$, the statement follows from the
case-by-case analysis in the proof of  \cite[Theorem
  5.3]{liuosserman}.  Statement (1) follows.

\bigskip\noindent The group $N_{n,a}$ contains the elements
$\alpha_{1,n,a}, \alpha_{2,n,a}, \alpha_{3,n,a}$ from Lemma
\ref{lem:elts}(4). Recall that $G_{n,a}$ is a transitive group on
$d^n$ letters. In particular, it follows that conjugation by $G_{n,a}$
acts transitively on the $d^{n-1}$ blocks of $d$ indices, where the
$i$th block corresponds to the vertices of $T_n^i$. It therefore
suffices to prove Statement (2) for $i=1$.

 Replacing $\alpha_{j,n,a}$ by a conjugate under $G_{n,a}$, we may
 assume that the component $\alpha_{j,n,a}^1$ in $\mathcal{S}(T_n^1)$
 is non-trivial. Note that the group elements we conjugate by may
 depend on $j\in \{1,2,3\}$. We conclude that the image of $N_{n,a}$
 under projection to $\mathcal{S}(T_n^1)$ contains an $e_1$-cycle, an
 $e_2$-cycle, and an $e_3$-cycle.  We denote this image by
 $G_{n,a}^1$.

The group $G_{n,a}^1$ may be identified with the Galois
group of the specialization of $f$ at a point $b$ with
$f^{n-1}(b)=a$. The corresponding field extension $K_{1,b}$ of $K_{0,b} := \QQ$
defined by $f$ may therefore be identified with a subextension of
$K_{n,a}/K_{n-1,a}$ (cf. Definition \ref{def:Kn}). Since $[K_{n,a}:K_{0,a}]=d^n$ (Proposition
\ref{prop:Eisenstein}), it follows that $ [K_{1,b}:K_{0,b}]=d$. We
conclude that $G_{n,a}^1$ is a transitive group on $d$ letters. The
argument from the proof of Statement (1) shows that $G_{1, \QQb}\subseteq
G_{n,a}^1\subseteq S_d$. This proves Statement (2).
\end{proof}

Theorem \ref{G2aG2} below extends the conclusion of Proposition
\ref{prop:condEisenstein}(1) that $G_{1, \QQb}\subseteq G_{1,a}$ under Conditions \ref{conditions} to arbitrary $n$. Phrased differently,
we show that we do not need additional conditions on $a$ in passing from
$n=1$ to arbitrary $n$. We exclude  exactly the same types as
in Theorem \ref{thm:G2alt}.

\begin{theorem}\label{G2aG2}
Let $f$ be a normalized Belyi map of type $(d; e_1, e_2,
e_3)\notin\{(4;3,3,3), (6;4,4,5)\}$.  Choose $a \in
\mathbb{P}^1(\QQ)\setminus \{0,1,\infty\}$ and distinct prime numbers
$p, q_1, q_2, q_3$ such that Conditions \ref{conditions} hold.  Then
\[
G_{n, \QQb}\subseteq G_{n,a} \qquad\text{ for all }\qquad n \geq 2.
\]
\end{theorem}

\begin{proof}
We argue as in the proofs
of Section \ref{sec:gensys}.  The statement of the theorem for $n=1$ holds by
Proposition \ref{prop:condEisenstein}(1).

 After replacing the $h_{j,1}$ by a
conjugate under $G_{1,a}\supset A_d$ we may assume that
\begin{equation}\label{eq:supp}
|\supp(h_{1,1})|\cap |\supp(h_{2,1})|\cap|\supp(h_{3,1})|=1.
\end{equation}
Here we have also used that $2d+1=e_1+e_2+e_3$. (For the generating
system  from Lemma \ref{lem:gensys}(2), the unique
integer in the intersection is $e_3$.)  In the case that
$G_{1,\QQb}\simeq S_d$ we have that $G_{1,\QQb}=G_{1,a}\simeq S_d$, and
in this case we may even assume that $h_{j,1}=g_{j,1}$ for all $j\in
\{1,2,3\}$. More precisely, from the fact that the generating system
 from Lemma \ref{lem:gensys}(2) is weakly
rigid it follows that we may assume that
$h_{j,1}=\sigma^{-1}g_{j,1}\sigma$ for all $j\in \{1,2,3\}$ and some
$\sigma\in S_d$ which is independent of $j$.

   The proof of Lemma
\ref{alphas}(2) may be applied to the elements $\alpha_{j,n,a}$ defined in
Lemma \ref{lem:elts}(4). In the current situation
we conclude that the component $\alpha_{j,n,a}^i$ of $\alpha_{j,n,a}$ in
$\mathcal{S}(T_n^i)$ satisfies
\[
\alpha_{j,n,a}^i=\begin{cases}\text{ an $e_1$-cycle }&\text{  if } i\in\supp(h_{h,1}),\\
\text{ trivial }&\text{ otherwise}.\end{cases}
\]

 The following claim is
analogous to Claim $1$ in the proof of Proposition \ref{prop:betas}.

\bigskip
\textbf{Claim }1: We have that $N_{n,a}^i\neq \emptyset$ for some
$1\leq i\leq d^n$.

Consider the commutator 
\[
\beta_{n, a} := [\alpha_{1,n,a},[\alpha_{2,n,a}, \alpha_{3,n,a}]].
\] 
Since
$\alpha_{2,n,a}, \alpha_{3,n,a} \in N_{n,a}$, it follows from
Equation \eqref{eq:wrconj} that the component $\beta_{n, a}^i$ of
$\beta_{n, a}$ in $\mathcal{S}(T_n^i)$ is
\[
\beta_{n, a}^i = [\alpha_{1,n,a}^i,[\alpha_{2,n,a}^i, \alpha_{3,n,a}^i ] ]\in
  N_{n,a}^i.
\] 
Arguing as in the proofs of Lemma \ref{alphas} and Proposition
\ref{prop:betas} we conclude that the component $\beta_{n, a}^i$ for
$i=1+\sum_{k=1}^{n-1}(\ell_k-1)d^{k-1}$ is non-trivial if and only if
$\ell_k \in \supp(h_{1,1})\cap\supp(h_{2,1})\cap \supp(h_{3,1})$ for
all $1\leq k\leq n-1$.  Equation (\ref{eq:supp}) implies that
$\beta_{n, a}^i$ is non-trivial for a unique $i_0$, i.e., $\beta_{n, a}\in
N_n^{i_0}$. Moreover,  the $i_0$-th component
$\beta_{n, a}^{i_0}$ equals
\[
[h_{1,1},[h_{2,1},h_{3,1}]]\neq \mathrm{id}.
\]
The last statement may for example be checked explicitly using the
fact that $(h_{1,1}, h_{2,1}, h_{3,1})$ is uniformly conjugate under
$S_d$ to $(g_{1,1}, g_{2,1}, g_{3,1})$, together with the explicit
expression for the $g_{j,1}$ given in Lemma \ref{lem:gensys}(2).
Claim 1 is now proved.

\bigskip
\textbf{Claim }2: $N_{n,a}^i\simeq A_d$ for all $1\leq i\leq d^{n-1}$.\\

Claim 2 follows from Claim 1  as in the proof of
Proposition~\ref{prop:betas} (Claim~2). 

The proof of Lemma~\ref{lem:conjH}
applies in this case as well: in the proof of that lemma, we only use
the fact that exactly $(e_j)^{n-1}$ components of the element $\alpha_{j, n}\in
N_{n, \QQb}$ are non-trivial. In the current situation, this property follows
from Lemma \ref{lem:elts}(4).
We conclude  that $N_{n,a}$ contains a
subgroup isomorphic to~$N_{n,\QQb}$. The statement of the theorem
follows as in the proof of Theorem~\ref{thm:G2alt}.
\end{proof}

The following straightforward corollary of Theorem \ref{G2aG2}
summarizes the relation between the groups $G_{n, \QQb}$, $G_{n,a}$,
$G_{n, \QQ}$ in the case that all assumptions we have imposed at
various places in this paper hold. Recall that we gave explicit
conditions guaranteeing that $G_{n, \QQb}=G_{n,\QQ}$ for all $n\geq 1$
in Corollaries~\ref{cor:condsgn} and \ref{cor:descent}.

\begin{corollary}\label{cor:index} Let $f$ be a normalized Belyi map of type $(d; e_1, e_2, e_3)
\notin\{(4;3,3,3), (6;4,4,5)\}$. Choose $a \in
    \mathbb{P}^1(\QQ)\setminus \{0,1,\infty\}$ and distinct primes
     $p, q_1, q_2, q_3$ such that Conditions \ref{conditions}
    hold. Assume that $G_{n, \QQb}=G_{n,\QQ}$ for all $n\geq 1$. Then
    $G_{n,a}=G_{n,\QQb}$ for all $n\geq 1$. 
\end{corollary}

\begin{proof}
We have $G_{n, \QQb}\subseteq G_{n,a} \subseteq G_{n, \QQ}\subseteq
\Aut(T_n)$ for all $n \geq 1$: the first inclusion is Theorem
\ref{G2aG2}, the second one is Equation (\ref{eq:Galin}, and the third
holds by definition. The assumption that $G_{n,\QQb}=G_{n, \QQ}$ for
all $n\geq 1$ therefore implies that all three groups are equal. The second statement on the index follows from Lemma \ref{lem:En}(1).
\end{proof}

\begin{remark}
Corollary \ref{cor:index} also immediately implies that, under the same assumptions,
\[
[\mathrm{Aut}(T_n) : G_{n,a}] \to \infty \text{ as } n \to \infty.
\]
\end{remark}

\section{Applications to dynamical sequences}\label{sec:app}

Let $f$ be a normalized Belyi map of degree $d$ such that $G_{1,
  \overline{\QQ}} \simeq S_d$. By Theorem \ref{thm:G2alt}, the
iterates~$f^n$ then have geometric monodromy groups $G_{n,
  \overline{\QQ}} \simeq E_n$ for all $n \geq 1$, where the groups
$E_n$ are defined in Definition \ref{def:En}.  We prove in Theorem
\ref{EnfixEn} that when $E_1 \simeq S_d$ and for any $d \geq 3$,  the proportion of elements of $E_n \simeq G_n$ that
fix a leaf on level $n$ tends to zero as $n$ tends to infinity.  This
generalizes \cite[Theorem 5.1]{BFHJY}, where the authors prove the
same for the Belyi map $f(x) = -2x^3+3x^2$ of degree $d=3$; our proof
is inspired by theirs.

Choose $a \in \mathbb{P}^1(\QQ)\setminus \{ 0, 1, \infty \}$ and distinct primes $p, q_1, q_2, q_3$ such that Conditions \ref{conditions} hold. Then $G_{n, \overline{\QQ}} \subseteq G_{n,a} \subseteq G_{n, \QQ}$ for all $n \geq 1$, by Theorem \ref{G2aG2}. Further assume that $G_{n, \overline{\QQ}} \simeq G_{n, \QQ}$ for all $n \geq 1$, so that these inclusions all become equalities. For $f$ of degree $d \geq 3$, 
Theorem \ref{EnfixEn} then shows that the proportion of elements of $G_{n,a}$ that fix a leaf tends to zero as $n$ tends to infinity.
In Corollary \ref{cor:seq} we derive some arithmetic dynamical consequences from this theorem. 

\subsection{Proportion of elements fixing a leaf}

\begin{definition}
Let $n \geq 1$. Define 
\[
\begin{split}
E_{n, \text{fix}} & = \{ \text{ elements of } E_n \text{ that fix a leaf } \},\\
A_{n, \mathrm{id}} & = \vert \{ \text{ elements of } E_n \text{ that act as } \mathrm{id} \text{ on } T_1 \} \vert , \\
A'_{n, \mathrm{id}} & = \vert \{ \text{ elements of } E_{n, \text{fix}} \text{ that act as } \mathrm{id} \text{ on } T_1 \} \vert ,
\end{split}
\]
and for $2 \leq k \leq d$, let  
\begin{equation}
\begin{split}
A_{n, k} & = \vert \{ \text{ elements of } E_n \text{ that act as as } \tau \text{ with } \vert \mathrm{supp}(\tau) \vert = k \text{ on } T_1 \} \vert, \\
A'_{n, k} & = \vert \{ \text{ elements of } E_{n,\text{fix}} \text{ that act as as } \tau \text{ with } \vert \mathrm{supp}(\tau) \vert = k \text{ on } T_1 \} \vert.
\end{split}
\end{equation}
\end{definition} 

 We see that $\vert E_n \vert = A_{n, \text{id}} + \sum_{i=2}^{d}
 A_{n,i}$ and $\vert E_{n, \text{fix}} \vert = A'_{n, \text{id}} +
 \sum_{i=2}^{d-1} A'_{n,i}$.  Recall from Section~2 that
\begin{equation}\label{eq:En+1}
\vert E_{n+1} \vert = \frac{\vert E_1 \vert}{2} \vert E_{n} \vert^{d} =  \vert E_{n} \vert^{d} d!/2.
\end{equation}
Finally, we note that for any $n \geq 1$, exactly half of the elements
of $E_{n-1}\wr E_1$ are contained in $E_n$ by Definition \ref{def:En}.

\begin{theorem}\label{EnfixEn}
Let $T_n$ be the regular $d$-ary rooted tree of level $n$. 
Then 
\[
\vert E_{n, \text{fix}} \vert / \vert E_n \vert \to 0 \text{ as } n \to \infty.
\]
\end{theorem}

\begin{remark}
Let $f$ be a normalized Belyi map of degree $d \geq 3$, 
such that $G_{n, \QQ} \simeq G_{n, \overline{\QQ}} \simeq E_n$ for all $n \geq 1$. If we choose $a \in \mathbb{P}^1(\QQ) \setminus \{0,1,\infty\}$  and distinct primes $p, q_1, q_2, q_3$ such that Conditions \ref{conditions} hold, then $G_{n, \overline{\QQ}} \simeq G_{n,a} \simeq G_{n, \QQ}$ for all $n \geq 1$, and Theorem \ref{EnfixEn} implies that
\[
\vert G_{n, a, \text{fix}} \vert / \vert G_n \vert \to 0 \text{ as } n \to \infty,
\]
where $G_{n,a, \text{fix}} = \{ \text{ elements of } G_{n,a} \text{ that fix a leaf } \}$.
\end{remark}

\begin{proof} 
By inclusion-exclusion, this yields
\begin{equation}\label{eq:An+1idfix}
A'_{n+1, \text{id}} = \frac{1}{2} \left( d \vert E_{n, \text{fix}} \vert \vert E_n \vert^{d-1} - \binom{d}{2} \vert E_{n, \text{fix}} \vert^2 \vert E_n \vert^{d-2} +\ldots + (-1)^{d-1} \vert E_{n, \text{fix}} \vert^d \right).
\end{equation}

Similarly, an element of $E_{n+1,\text{fix}}$ that acts as a non-trivial permutation $\tau$ on $T_1$ with $\vert \mathrm{supp}(\tau) \vert = k$ is of the form $((\sigma_1, \ldots, \sigma_d),\tau)$, where 
at least one of the $\sigma_j$ with $j~\not\in~\mathrm{supp}(\tau)$ is in $E_{n,\text{fix}}$. Moreover, $\sgn_2(\tau) \prod_{j=1}^d \sgn_2(\sigma_j) =~1$. This yields
\begin{equation}\label{eq:An+1kfix}
\begin{split}
A'_{n+1, k} \leq \frac{1}{2} \vert E_n \vert^k \cdot & \Big( (d-k) \vert E_{n, \text{fix}} \vert \vert E_n \vert^{d-k-1} \\ &- \binom{d-k}{2} \vert E_{n, \text{fix}} \vert^2 \vert E_n \vert^{d-k-2} +\cdots + (-1)^{d-k-1} \vert E_{n, \text{fix}} \vert^{d-k} \Big).
\end{split}
\end{equation}

Adding the contributions from \eqref{eq:An+1idfix} and \eqref{eq:An+1kfix}, dividing out by $\vert E_{n+1} \vert$, and using Equation \eqref{eq:En+1} yields 
\begin{equation}\label{eq:Anfix}
\begin{split}
\frac{\vert E_{n+1,\text{fix}} \vert}{\vert E_{n+1} \vert} \leq
\frac{1}{\vert E_1 \vert} & \Bigg( \left(1 - \left(1 - \frac{\vert
  E_{n, \text{fix}} \vert}{\vert E_n \vert } \right)^d \right) \\ &+
\sum_{i=2}^{d-1} \vert \{ \tau \in S_d : \vert
\mathrm{supp}(\tau)\vert = i \} \vert \cdot \left(1 - \left(1 - \frac{\vert
  E_{n, \text{fix}} \vert }{\vert E_n \vert} \right)^{d-i}
\right)\Bigg).
\end{split}
\end{equation}

For any $n \geq 1$, let $x_n$ denote $\vert E_{n,\text{fix}}\vert / \vert E_n \vert$. Then Equation \eqref{eq:Anfix} shows that $x_{n+1} \leq \phi(x_n)$, where 
\[
\phi: x \mapsto \frac{1}{d!} \sum_{i=0}^{d} \vert\{ \tau \in S_d : \vert \mathrm{supp}(\tau)\vert = i \}\vert \cdot \left(1 - \left(1 - x \right)^{d-i} \right).
\]
The function $g(x) =1- (1-x)^k$ is increasing on $[0,1]$ for any $k \geq 1$, hence so is $\phi$.
Since $1- (1-x)^k \geq 0$ for any $x\in [0,1]$ and any $k \geq 1$, we have $\phi(x) \geq 0$ for any $x \in [0,1]$. 

Hence we find that 
\[ 
x_{n+1} \leq \phi(x_n )\leq \ldots \leq \phi^n(x_1) \leq \phi^n(1),
\]
where $x_1=\vert E_{1,\text{fix}}\vert / \vert E_1 \vert \leq 1$. Let $y_n$ denote $\phi^n(1)$. Then $\{y_n\}_{n\geq 1}$ is a non-increasing sequence in the interval $[0,1]$. Hence it has a limit $y$ which satisfies $\phi(y)=y$. Since the only solution to $\phi(y)=y$ in $[0,1]$ is $y=0$, the limit of the $y_n=\phi^n(1)$ is zero. Thus, the limit of $x_n=\vert E_{n,\text{fix}}\vert / \vert E_n \vert$ is also zero, as required.
\end{proof}

\subsection{Dynamical sequences}

For a set $S$ of prime numbers, let $\delta(S)$ denote its natural
density if it exists. A \emph{dynamical sequence} in a field $K$ is a
sequence $\{c_i\}_{i \geq 0}$ with $c_i \in K$ such that $c_i =
f(c_{i-1})$ for some map $f: K \to K$. Prime divisors of entries of
such sequences were first studied using Galois theory by Odoni in
\cite{Odoni-prime}. For four particular quadratic maps $f$, Jones
\cite{jonesquad} shows the density of prime divisors in the dynamical
sequence for $f$ is zero; Gottesman and Tang \cite{gottang} show
non-zero densities can also occur for quadratic maps. More general
treatments of higher degree maps can be found in
e.g. \cite{fabergran}, \cite{ingsil}. 

\begin{corollary}\label{cor:seq}
Let $f$ be a normalized Belyi map such that $G_{n, \QQ} \simeq G_{n,
  \overline{\QQ}} \simeq E_n$ for all $n \geq 1$.
\begin{enumerate}
\item Choose $a \in \mathbb{P}^1(\QQ) \setminus \{0,1,\infty\}$ and
  distinct primes $p, q_1, q_2, q_3$, such that Conditions
  \ref{conditions} hold.  Consider the sequence $\{a_i\}_{i \geq 0}$
  where $a_0 \in \mathbb{P}^1(\mathbb{Q})\setminus \{0,1,\infty, a\}$
  and $a_{i} = f(a_{i-1})$ for all $i \geq 1$.  Then $\delta(S) = 0$,
  where $S$ is the set of primes $q \in \mathbb{Q}$ such that $a_i
  \equiv a \pmod q$ for some $i \geq 0$.
\item Let $K$ be the splitting field of $f$ and consider the $d-e_1$ non-zero preimages of zero under $f$, denoted $c_j \in K$ for $1 \leq j \leq d-e_1$. 
Suppose that for each $c_j$ there exist primes $\mathfrak{p}_j, \mathfrak{q}_{1,j}, \mathfrak{q}_{2,j}, \mathfrak{q}_{3,j}$ of $K$ such that the natural analogues of Conditions~\ref{conditions} hold. 
Now form the sequence
  $\{b_i\}_{i \geq 0} \subseteq \QQ$ where $b_0 \in
  \mathbb{P}^1(\mathbb{Q})\setminus \{0,1,\infty\}$, and $b_{i} = f(b_{i-1})$ for all $i \geq
  1$.  Then $\delta(T) = 0$, where $T$ is the set of rational primes $q \in
  \mathbb{P}$ such that $q \vert b_i$ for some $i \geq 0$.
\end{enumerate}
\end{corollary}

\begin{proof}
\begin{enumerate}
\item For any $n \geq 1$, the set $S_n$ of primes $q$ not dividing $a$ such that $a_i \equiv a \pmod q$ for some $0 \leq i \leq n-1$ is finite.
Further, if $a_i \equiv a \pmod q$ for some $i \geq n$, then the rational map $f^n(x) - a$ has a rational root over $\mathbb{Z}/q\mathbb{Z}$ (namely, $a_{i-n}$). 
It follows from the Chebotarev density theorem that
\[
\delta(S) \leq \delta( \{ q \in \mathbb{P}: q \not\in S_n \text{ and } f^n(x) - a \text{ has a root modulo } q \} ) = \frac{\vert E_{n, \text{fix}}\vert}{\vert E_n \vert},
\]  
so the result follows by letting $n \to \infty$ and using Theorem \ref{EnfixEn}.
\item Consider the sequence as a subset of $K$. We may ignore the finitely many primes of bad or good inseparable reduction for $f$, as well as the finitely many primes dividing $b_0$. Then $b_i \neq 0$ for any $i \geq 0$, i.e., $b_i \neq c_j$ for any $1 \leq j \leq d-e_1$, since by our assumption $G_{n,\QQ} \simeq G_{n,\QQb} \simeq E_n$ none of the $c_j$ are rational.
Let $\mathfrak{p}$ be a prime of $K$ such that $b_i \equiv 0 \pmod{\mathfrak{p}}$ for some, without loss of generality minimal, value of $i \geq 1$. Writing $f(x) = x^{e_1}f_1(x)/f_2(x)$ as before, we see that $f_1(b_{i-1}) \equiv 0 \pmod{\mathfrak{p}}$ and hence that $b_{i-1}$ is congruent modulo $\mathfrak{p}$ (but not equal) to $c_j$ for some $1 \leq j \leq d-e_1$. 

Arguing as in Section~\ref{sec:spec}, the assumptions on the $c_j$ guarantee that $G_{n,c_j,K} \simeq G_{n,K} \simeq G_{n,\overline{\mathbb{Q}}} \simeq E_n$ for all $n \geq 1$. Therefore, we conclude by observing that 
\[
\delta(T) = \delta(\{ p \in \mathbb{P} : \exists \mathfrak{p}\mid p \text{ prime of } K \text{ such that }  b_i \equiv c_j \pmod{\mathfrak{p}} \text{ for some } i \geq 1 \text{ and } 1 \leq j \leq d-e_1  \})
\]
and arguing as in (1). 
\end{enumerate}
\end{proof}

\begin{remark} Corollary \ref{cor:seq} generalizes \cite[Proposition 6.1 and Corollary 6.2]{BFHJY} to normalized Belyi maps of degree $ d \geq 3$ whose Galois groups satisfy $G_{n, \QQ} \simeq G_{n, \overline{\QQ}} \simeq E_n$.
\end{remark}

\providecommand{\bysame}{\leavevmode\hbox to3em{\hrulefill}\thinspace}
\providecommand{\MR}{\relax\ifhmode\unskip\space\fi MR }
\providecommand{\MRhref}[2]{%
  \href{http://www.ams.org/mathscinet-getitem?mr=#1}{#2}
}
\providecommand{\href}[2]{#2}

\end{document}